\documentclass[dvipsnames]{article}
\usepackage[utf8]{inputenc}
\usepackage{amsmath}
\usepackage{amsfonts}
\usepackage{graphicx}
\usepackage{algorithm}
\usepackage[noend]{algpseudocode}
\usepackage{amsthm}
\usepackage{xcolor}
\usepackage{tikz}
\usepackage{hyperref}
\usepackage{xcolor}
\usepackage{bbm,mathtools}

\newcommand{\bs}[1]{\boldsymbol{#1}}
\newcommand{\R}{\mathbbm{R}}
\newcommand{\N}{\mathbbm{N}}

\newtheorem{proposition}{Proposition}

\usepackage{float}
\usepackage[backend=biber,style=numeric, maxnames=99, sorting=none]{biblatex}
	
\usepackage{soul}
\usepackage[noabbrev]{cleveref}

\sethlcolor{yellow}

\renewcommand\hl[1]{#1}

\usepackage[left=1.25in,right=1.25in,top=1in]{geometry}

\bibliography{refs}
\DeclareMathOperator*{\argmax}{\arg\!\max}

\usepackage{makecell}
\usepackage{subfigure}
\usepackage{authblk}
\usepackage{algorithm}
\usepackage{algpseudocode}

\usepackage{booktabs}
\usepackage{pgf-pie} % version: https://github.com/pgf-tikz/pgf-pie Dec 26, 2020
\usepackage{tikzscale}
\usepackage{caption}

\usepackage[ group-separator={,} ]{siunitx}
\usepackage[referable]{threeparttablex}

\newcolumntype{L}{@{}l@{}} % a left column with no intercolumn space on either side
\newcommand{\mc}[1]{\multicolumn{1}{c}{#1}} % shorthand macro for column headings

\title{Adaptive Density Tracking by Quadrature\\ for Stochastic Differential Equations}
\author[1]{Ryleigh A. Moore\thanks{Corresponding author
email: rmoore@math.utah.edu\\
R. Moore and A. Narayan were partially supported by AFOSR under award FA9550-20-1-0338. A. Narayan is partially supported by NSF DMS-1848508.}}
\author[1]{Akil Narayan}
\affil[1]{Scientific Computing and Imaging Institute, and Department of Mathematics,
University of Utah}
\begin{document}
\maketitle
\begin{abstract}
  \noindent Density tracking by quadrature (DTQ) is a numerical procedure for computing solutions to Fokker-Planck equations that describe probability densities for stochastic differential equations (SDEs). In this paper, we extend upon existing trapezoidal quadrature rule DTQ procedures by utilizing a flexible quadrature rule that allows for unstructured, adaptive meshes. We describe the procedure for $N$-dimensions, and demonstrate that the resulting adaptive procedure can be significantly more efficient than the trapezoidal DTQ method. \hl{We show examples of our procedure for problems ranging from one to five dimensions}.\\ %Although we consider two-dimensional examples, all our computational procedures are extendable to higher dimensional problems.\\
  
\end{abstract}
\section{Problem History and Background}
Stochastic differential equations (SDEs) are prevalent in many areas of research. Kloeden and Platen \cite{kloeden1992applications} outline a variety of SDE uses, including population dynamics, protein kinetics, psychology problems involving neuronal activity, investment finance and option pricing, turbulent diffusion of a particle, radio-astronomy and the analysis of stars, helicopter rotor and satellite orbit stability, biological waste treatment with analysis of air and water quality, seismology and structural mechanics, the stability of materials prone to fatigue cracking, and blood clotting dynamics and cellular energetics. In this paper, we are interested in solving SDEs and their associated Fokker-Planck equations. 

\subsection{Stochastic Differential Equations}
\label{SDESection}
Let $\mathbf{W}_t$ be an $N$-dimensional Wiener Process and $\mathbf{X}_t$ be an $N$-dimensional vector stochastic It$\widehat{\text{o}}$ diffusion process governed by the SDE
\begin{equation}
\label{eqn:SDE}
d\mathbf{X}_t=\textbf{f}(\mathbf{X}_t, t)dt+\mathbf{g}(\mathbf{X}_t, t)d\mathbf{W}_t
\end{equation}
with the drift $\textbf{f}(\mathbf{X}_t,t)$ as an $N$-dimensional vector and the diffusion defined by an $N \times N$-dimensional matrix $\textbf{g}(\mathbf{X}_t, t)$. This equation is endowed with a $t=0$ initial condition $\bs{X}_0$.

%\subsection{Fokker-Planck Equation}
The evolution of the probability density function for $\mathbf{X}_t$ is governed by the corresponding Fokker-Planck partial differential equation (PDE).
%The Fokker-Planck equation is a PDE  describing the evolution of the probability density of a particle's location assuming Gaussian white noise, drift, and diffusion acting on the particle.
The Fokker-Planck equation for the evolution of the probability density $p(\mathbf{x},t)$ of the random variable $\mathbf{X}_t$ from \cref{eqn:SDE} is given by
\begin{equation}
    \frac{\partial}{\partial t}p(\mathbf{x},t) = -\sum_{i=1}^N \frac{\partial}{\partial x^{(i)}}\left[\mathbf{f}_i(\mathbf{\mathbf{x}},t)p(\mathbf{x},t)\right] + \sum_{i, j=1}^N \frac{\partial^2}{\partial x^{(i)} \partial x^{(j)}}\left[\mathbf{D}_{i,j}(\mathbf{x},t)p(\mathbf{x},t)\right]
    \label{eqn:FPE}
\end{equation}
where $\mathbf{x} = (x^{(1)}, \dots, x^{(N)})^T$.
The diffusion tensor $\bs{D}$ is related to the SDE diffusion $\bs{g}$ by
\begin{equation*}
    \mathbf{D}_{i,j}(\mathbf{x},t) = \frac{1}{2}\sum_{\ell=1}^N \mathbf{g}_{i,\ell}(\mathbf{x},t)\mathbf{g}_{j,\ell}(\mathbf{x},t),
\end{equation*}
see, e.g., \cite[p.~5]{Risken}.
Since $p(\mathbf{x},t)$ is a probability density function, it satisfies a normalization condition
\begin{equation*}
    \int_{\mathbb{R}^N} p(\mathbf{x},t)d\mathbf{x} =1.
\end{equation*}
This paper focuses on numerical approximation of the time-dependent probability density function $p(\mathbf{x},t)$ governed by \cref{eqn:FPE}.

\subsection{Current Methods}\label{CurrentMethods}
Several methods have been developed to numerically approximate the statistics of the SDE's solution $\mathbf{X}_t$, or the probability density function $p(\mathbf{x},t)$ from the associated Fokker-Planck PDE in \cref{eqn:FPE}. Perhaps among the more straightforward approaches is through Monte Carlo simulation \cite{platen2010monte}, which typically collects a large ensemble of realizations of $\bs{X}_t$ from \cref{eqn:SDE}. The large number of samples needed to sufficiently approximate the solution of the SDE makes using this method with sufficient accuracy computationally expensive.

Many numerical methods have been developed to compute solutions to the Fokker-Planck equation, such as finite element methods (FEMs) \cite{Pichler, Bergman1982, Langley1985, spencer1993numerical, masud2004multiscale, kumar2006partition, wojtkiewicz2000numerical,Wojtkiewicz} and finite difference methods (FDMs) \cite{Pichler, Wojtkiewicz, Kumar2006}. FEMs are often preferable over FDMs to solve the Fokker-Plank equation because of their accuracy and stability; however, they can be more complicated to implement compared to FDMs. Current FDMs are empirically less numerically stable than FEMs, but they also usually require less memory and computational power to implement \cite{Pichler, Wojtkiewicz}. Both FEMs and FDMs suffer from the curse of dimensionality stemming from the computational difficulty of forming a sufficiently dense mesh in $N$ dimensions. When using FDMs or FEMs, erroneous oscillations and negative values can arise if the drift is large compared to the diffusion. One method to address this challenge utilizes a moving finite element mesh where basis functions, which depend on time instead of only on space, are used to eliminate the spurious oscillations  \cite{Harrison1988}. Some adaptive FEM procedures monitor regions of non-negligible probability and adjust the mesh coarseness appropriately \cite{cotter2013adaptive}. Adaptive FEM procedures also adjust the mesh based on the local value and gradient of $p$ near boundary regions \cite{Razi2011}. 

Additionally, finite volume methods (FVMs) have been applied to the conservation form of the Fokker-Plack equation, utilizing a linear multistep method for temporal discretization \cite{Ferm2004}. Such procedures can be made to adaptively adjust the mesh and time step based on an error tolerance criterion.  
%Furthermore, adaptive meshes are utilized in numerical path integration methods, also called transformed path integral methods \cite{Subramaniam2017}. 
%These methods propagate the grid in a Lagrangian way relative to a given fixed grid in a transformed space. 
Deep learning approaches have also been leveraged to numerically approximate solutions to the Fokker-Planck equation. If a large amount of training data is available, neural networks can be used to learn solution behavior \cite{xu2020ML}. Of course, this requires availability of such training data, and guaranteeing generalizability and accuracy with such approaches is often difficult. 

The curse of dimensionality is a concern with current methods because the required memory and computational cost increases substantially with the dimension $N$ of the problem. FEMs and finely discretized FDMs have been used to solve four-dimensional problems \cite{wojtkiewicz2000numerical, Wojtkiewicz}, but more work is needed for higher dimensional problems to become tractable.
%\akil{Meeting:Added background}

In this research, we extend current density tracking by quadrature methods to approximate the PDF of SDEs in high dimensions. DTQ has also been described previously as numerical path integration (NPI) \cite{wehner1983numerical, NAESS1993, Yu1997} and has been applied to many different disciplines including engineering \cite{NAESS1993} and finance \cite{RosaClot1999, Skaug2007, Linetsky1997}. A transformed path integral approach is discussed in \cite{Subramaniam2017}. Results on the stability, consistency, and convergence of NPI/DTQ under certain conditions are given in \cite{chen2018, Bhat2018}. In one dimension, DTQ has been shown to be a convergent method that computes an approximation to the probability density function $p(\mathbf{x},t)$ of $\mathbf{X}_t$ on a discrete grid. In some examples, DTQ is 100 times faster compared to other methods with similar accuracy \cite{Bhat2018}, making it very appealing for further study.

\subsection{Outline and Contributions of this Paper}

We work to augment current DTQ algorithms by implementing an accurate and flexible interpolatory quadrature rule, along with adaptive mesh updates, to minimize the computational cost. Our quadrature rule allows for an $N$-dimensional, unstructured mesh that provides flexibility to allocate mesh points to areas of high density and to remove mesh points from areas of low density. The unstructured mesh allows for nontensorial discretizations and partially addresses the curse of dimensionality. 

We first summarize the current DTQ method which utilizes a structured mesh and a trapezoidal quadrature rule \cite{Bhat2018}. Then, we discuss and fully detail our adaptive DTQ method, which uses an interpolatory quadrature rule on Leja points. Finally, we compare the two procedures.

\section{Density Tracking by Quadrature}
We present DTQ in the framework of $N$-dimensional SDEs in \cref{eqn:SDE}. For a fixed temporal step size $h > 0$, we first discretize the SDE in \cref{eqn:SDE} in time using the Euler-Maruyama method, which results in the equation
 \begin{equation}
    \label{eqn:E-M}
     \mathbf{\widetilde{X}}_{n+1}=\mathbf{\widetilde{X}}_n+\mathbf{f}(\mathbf{\widetilde{X}}_n,t)h+\mathbf{g}(\mathbf{\widetilde{X}}_n,t)\sqrt{h}\mathbf{Z}_{n+1}.
\end{equation}
Here, $\mathbf{\widetilde{X}}_n$ represents an approximation of the state $\mathbf{X}_t$ at time $t_n=nh$ and $\mathbf{Z}_{n+1}$ is a standard $N$-dimensional normal random variable (i.e., 0 mean, identity covariance). 

Now, we interpret the time discretized \cref{eqn:E-M} as a discrete-time Markov chain. Let $\widetilde{p}(\mathbf{x},t_n)$ denote the PDF at location $\mathbf{x}$ at time $t_n$ of the Markov chain. From \cref{eqn:E-M}, we observe that the conditional density of $\mathbf{\widetilde{X}}_{n+1}$ given $\mathbf{\widetilde{X}}_n=\mathbf{y}$ is Gaussian with mean $\boldsymbol{\widetilde \mu}=\mathbf{y}+\mathbf{f}(\mathbf{y})h$ and covariance $\boldsymbol{ \widetilde \Sigma} = h\mathbf{g(y)g(y)}^T$, notated as

\begin{equation*}
P(\mathbf{\widetilde{X}}_{n+1}=\mathbf{x}~|~\mathbf{\widetilde{X}}_n=\mathbf{y})
:=G(\mathbf{x},\mathbf{y}; \mathbf{\widetilde \mu} ,\mathbf{\widetilde \Sigma}),
%\label{eqn:condProb}
\end{equation*}
where
\begin{equation}
\label{eqn:Gxy}
G(\mathbf{x},\mathbf{y}) :=
G(\mathbf{x}, \mathbf{y};\boldsymbol{\widetilde \mu} ,\boldsymbol{\widetilde \Sigma})) :=\frac{1}{\sqrt{(2\pi)^{N}|\boldsymbol{\widetilde\Sigma}|}} \exp \left(-\frac{1}{2}\left(\mathbf{x}-\boldsymbol{\widetilde \mu}\right)^T \mathbf{\widetilde \Sigma}^{-1}\left(\mathbf{x}-\boldsymbol{\widetilde \mu}\right)\right).
\end{equation}
Notice that $G(\mathbf{x},\mathbf{y})$ depends on the drift $\mathbf{f}$ and diffusion $\mathbf{g}$ through $\boldsymbol{\widetilde \mu}$ and $\boldsymbol{\widetilde \Sigma}$, but we will omit this dependence notationally.

The evolution of the density of the Markov chain is described by the associated Chapman-Kolmogorov equation,
\begin{equation}
    \begin{aligned}
 \widetilde{p}(\mathbf{{x}},t_{n+1})&=\int_{\mathbb{R}^N} P(\mathbf{\widetilde{X}}_{n+1}=\mathbf{x}~|~\mathbf{\widetilde{X}}_n=\mathbf{y})\widetilde{p}(\mathbf{y},t_n)d\mathbf{y}\\
 &=\int_{ \mathbb{R}^N} G(\mathbf{x},\mathbf{y})\widetilde{p}(\mathbf{y},t_n)d\mathbf{y}.
\end{aligned}
\label{eqn:C-K}
\end{equation}

The next step is to approximate the evolution of $\widetilde{p}(\mathbf{x},t_n)$ by discretizing \cref{eqn:C-K} in space. Let $\{\mathbf{y}_1, \dots, \mathbf{y}_s\}$ be a set of global mesh points where we will track the density. Then, $\widetilde{p}$ can be approximated by $\widehat{p}$ via a discretization of \cref{eqn:C-K},
\begin{equation}
\label{eqn:DiscreteC-K}
    \begin{aligned}
      \widehat{p}\left( \mathbf{y}_j, t_{n+1} \right) = \sum_{i=1}^m G(\mathbf{y}_j,\boldsymbol{\eta}_i)\widehat{p}(\boldsymbol{\eta}_i,t_{n})\omega_i
\end{aligned}
\end{equation}
where $\{\omega_i\}_{i=0}^m$ are quadrature weights and $\{\boldsymbol \eta_i\}_{i=0}^m$ are quadrature nodes. The initial condition is given as $\widehat{p}(\mathbf{y}_j, 0) = \widetilde{p}(\mathbf{y}_j, 0)$. 

One-dimensional DTQ has previously been analyzed and error estimates were established \cite{Bhat2018, chen2018}. Convergence of $p$ to $\widetilde{p}$, when using the Euler-Maruyama method \cite{BALLYTALAY+1996+93+128}, was used to help show that, for one-dimensional DTQ, the density $\widehat{p}$ converges in $L_1$ exponentially to the exact density of the Markov chain $\widetilde{p}$, and $\widehat{p}$ converges to the exact
density of $p$ with a first-order convergence rate \cite{Bhat2018}. Furthermore, DTQ has been used for parameter inference problems \cite{Bhat2018a}, and a two-dimensional implementation was employed to analyze basketball tracking data from the National Basketball Association \cite{bhat20162DTQ}.

\subsection{Simple One-Dimensional Interpretation of DTQ}
\label{1DTrapDTQ}
One way to approximate \cref{eqn:DiscreteC-K} is using a trapezoidal quadrature rule \cite{Bhat2018}. We consider a one-dimensional problem with an equispaced mesh, $\{\widehat{y}_1, \dots, \widehat{y}_q\} \subset \mathbb{R}$, which has a spatial step size $\kappa$. 
%\akil{Meeting: There used to be $s$ points called $y$, but now there are $n$ points called $x$? And $n$ used to be a time index and now it's something else?}
Then, the density at a point $\widehat{y}_j \in \{\widehat{y}_i\}_{i=1}^q$ is updated using the trapezoidal quadrature rule
\begin{equation}
    \widehat{p}(\widehat{y}_j,t_{n+1}) = \kappa \sum_{i=1}^q G(\widehat{y}_j,\widehat{y}_i)\widehat{p}(\widehat{y}_i,t_{n}).
    \label{eqn:TransitionMat}
\end{equation}
Mathematically, \cref{eqn:TransitionMat} can also be written as the matrix vector multiply
\begin{equation*}
    \mathbf{{P}}_{n+1}= \kappa\mathbf{G}\mathbf{P}_n
   %\label{eqn:TensorizedDTQ}
\end{equation*}
where $\mathbf{P}_{n+1} = \left[\widehat{p}(\widehat{y}_1,t_{n+1}), \dots, \widehat{p}(\widehat{y}_q,t_{n+1})\right]^T$ and $\mathbf{G}_{i,v}= G(\widehat{y}_i,\widehat{y}_v)$.

\subsection{Tensorized DTQ}\label{ssec:tensor-DTQ}
In more than one dimension, $N > 1$, a straightforward choice for the mesh is a tensorial grid, e.g., an isotropic grid is formed from the tensorization of a univariate grid,
\begin{align*}
  \left\{ \mathbf{\widehat{y}}_i \right\}_{i=1}^s &= \bigotimes_{\ell=1}^N \left\{ \widehat{y}_1, \ldots, \widehat{y}_q \right\}, & \left\{ \widehat{y}_1, \ldots, \widehat{y}_q \right\} &\subset \mathbb{R}.
\end{align*}
In this case, the discretization of \cref{eqn:DiscreteC-K} can proceed dimension by dimension. If the univariate grid $\{\widehat{y}_i\}_{i=1}^q$ is equispaced with mesh stepsize $\kappa > 0$, then \cref{eqn:DiscreteC-K} can be written as
\begin{equation*}
    \begin{aligned}
\widehat{p}(\mathbf{\widehat{y}}_j,t_{n+1}) = \kappa^N\sum_{i=1}^s G(\mathbf{\widehat{y}}_j,\mathbf{\widehat{y}}_i)\widehat{p}(\mathbf{\widehat{y}}_i,t_{n}).
\end{aligned}
\end{equation*}
In vector form, the above is
\begin{align*}
  \mathbf{P}_{n+1}&= \kappa^N\mathbf{G}\mathbf{P}_n, & \mathbf{G}_{i,v} &= G(\mathbf{\widehat{y}}_i, \mathbf{\widehat{y}}_v),
\end{align*}
where $\mathbf{P}_{n+1} \coloneqq \left[\widehat{p}(\mathbf{\widehat{y}}_1,t_{n+1}), \dots, \widehat{p}(\mathbf{\widehat{y}}_s,t_{n+1})\right]^T$. The matrix $\kappa^N\mathbf{G}$ contains values describing the movement of density; however, it is not a Markov transition matrix in general \cite{Bhat2018}.

The numerical solution $\widehat{p}$ can, in principle, be directly computed using this procedure; however, this can become expensive quickly as the dimension increases. For higher dimensional problems, we require $q^N$ mesh points for the tensorization strategy. For example, if we need $100$ points per dimension, in four dimensions we will need $10^8$ points, which is computationally prohibitive. In order to help extend DTQ to higher dimensions, we provide an a different strategy to discretize the integral in \cref{eqn:C-K}, which allows for an unstructured set of mesh points.

\section{DTQ on an Unstructured Mesh}
We will now describe our procedure for implementing DTQ on an unstructured mesh in $N$ dimensions. We utilize an unstructured, adaptive mesh and an interpolatory quadrature rule to approximate the integral in \cref{eqn:C-K} by treating a portion of the integrand as a Gaussian density. For each point in the global unstructured mesh, $\mathbf{y}_j \in \{\mathbf{y}_1, \dots, \mathbf{y}_s\}$,
we compute quadrature nodes and weights for an interpolatory quadrature rule. We denote the quadrature nodes as $\{\boldsymbol \eta_1, \dots, \boldsymbol \eta_m\}$, where we suppress the $j$ dependence since the procedure updates one mesh point at a time (eg.
$\{\boldsymbol \eta_1, \dots, \boldsymbol \eta_m\}=\{\boldsymbol \eta_1, \dots, \boldsymbol \eta_m\}_j$).

Now we will update the density associated to a member of the global mesh $\mathbf{y}_j$,
     \begin{align}
      \widetilde{p}(\mathbf{y}_j, t_n)
      &= \int_{\mathbb{R}^N} G(\mathbf{y}_j,\mathbf{y})\widetilde{p}(\mathbf{y},t_n)d\mathbf{y}\label{eqn:GptildeIntegral}\\
      \label{eqn:LaplaceFit}&= \int_{\mathbb{R}^N} r(\mathbf{y})\mathcal{N}(\mathbf{y};\boldsymbol{\mu},\boldsymbol{\Sigma})d\mathbf{y}
      \end{align}
so that 
      \begin{equation}
          \label{eqn:QuadRuleEta} \widehat{p}(\mathbf{y}_j,t_n)=\sum_{i=1}^m r(\boldsymbol{\eta}_i) \widehat{w}_{i}
      \end{equation}
     where 
    \begin{equation*}
    \label{mathcalN}
        \mathcal{N}(\mathbf{x}; \boldsymbol{\mu}, \boldsymbol{\Sigma})= \frac{1}{\sqrt{\pi^N| \boldsymbol{\Sigma}|}} \exp{\left(-\left(\mathbf{x}-\boldsymbol{\mu}\right)^T \boldsymbol{\Sigma}^{-1} \left(\mathbf{x}-\boldsymbol{\mu}\right)\right)}.
    \end{equation*}
\Cref{QuadFit}, describes how we effect the integral in \cref{eqn:LaplaceFit} by using a Laplace approximation of the integrand in \cref{eqn:GptildeIntegral} to identify $\boldsymbol{\mu}$ and $\boldsymbol \Sigma$ which subsequently allows us to define the weight function $\mathcal{N}$ (through $\boldsymbol{\mu}$ and $\boldsymbol \Sigma$) and the new integrand $r$. 
%\akil{You are just identifying $\mu$, and $\Sigma$.}
These values differ for each $\mathbf{y}_j$ (i.e., $\boldsymbol{\mu} = \boldsymbol{\mu_j},  \boldsymbol{\Sigma} = \boldsymbol{\Sigma_j}$, $r=r_j$, but we suppress this $j$ dependence). Then, \cref{ssec:iquad} details the identification of the quadrature weights in \cref{eqn:QuadRuleEta}, and \cref{LejaPointsSection} details the selection of the quadrature nodes.

\subsection{Laplace Approximation via Least Squares}
\label{QuadFit}
In this section, we describe how $r$ and $\mathcal{N}$ in \cref{eqn:LaplaceFit} are determined. 
We identify $\mathcal{N}$ as a Laplace approximation to the integrand of \cref{eqn:GptildeIntegral}, which we implement practically by performing a local least-squares quadratic fit to the log-integrand using nearby mesh points.

The Laplace approximation is computed for each point in the mesh. We consider a specific global mesh point $\mathbf{y}_j$ for this discussion. Let $\mathfrak{N}$ be the set of points used for the Laplace approximation. We note that $\mathfrak{N}$ depends on the point $\mathbf{y}_j$ (i.e. $\mathfrak{N}$ = $\mathfrak{N}_j$) but we suppress the $j$ dependence. When available, the quadrature nodes $\{\boldsymbol \eta_1, \dots, \boldsymbol \eta_m\}$, which were used at the previous time step to update the density at $\mathbf{y}_j$, are used. Otherwise, we use a set of nearest neighbor points to $\mathbf{y}_j$ (including itself). When needed, the nearest neighbors are determined using the Euclidean distance. The size of the set $\mathfrak{N}$ is formalized in \cref{sec:results}. 
%\akil{Meeting: What nearest neighbors? You haven't talked about how this is defined. Does this depend on $j$?}
In the beginning stages of the procedure, quadrature nodes are not known, so we use $\mathbf{y}_j$'s nearest neighbors. After quadrature nodes are known, we typically use them in place of the nearest neighbors. In this section, we will assume the use of $\{\boldsymbol \eta_1, \dots, \boldsymbol \eta_m\}$ for notational simplicity; however, the procedure is equivalent if nearest neighbors are used instead. 
%\akil{I don't understand why you can't just say that that the $\eta_i$ points are initialized as $\mathfrak{N}$.}

%\akil{Below you jump into defining $\psi$ without reminding the reader about what you're trying to do.}
Now, we will compute the Laplace approximation. Let the $i^{th}$ component of the vector $\boldsymbol \psi$ be given as
\begin{equation*}
\boldsymbol{\psi}_{i} := -\log(G(\mathbf{y}_j,
\boldsymbol{\eta}_i)\widehat{p}(\boldsymbol \eta_i,t_{n})), \hspace{1cm} i=1,\dots, m
\end{equation*}
so that $\boldsymbol{\psi}$ is an $m$-dimensional vector. The Laplace approximation will model this log-integrand as a quadratic polynomial,
\begin{equation}
 \boldsymbol{\psi}_{i} \approx {\widetilde{\psi}}(\boldsymbol{\eta}_i) \coloneqq c + \mathbf{d}^T \boldsymbol{\eta}_{i} + (\boldsymbol{\eta}_{i})^T \mathbf{A} \boldsymbol{\eta}_{i},
\label{Logapprox}
\end{equation} 
for a scalar $c$, vector $\mathbf{d} \in \mathbb{R}^N$, and a symmetric matrix $\mathbf{A} \in \mathbb{R}^{N \times N}$ that we identify via least-squares polynomial approximation. 

To describe this procedure, we require more notation.
Let $\alpha \in \N_0^N$ be a multi-index with the standard convention,
\begin{align*}
  \alpha &= (\alpha_1, \ldots, \alpha_N), & |\alpha| &\coloneqq \sum_{\ell=1}^N \alpha_\ell, & \boldsymbol{\eta}^\alpha = \prod_{\ell=1}^N \left( \eta^{(\ell)} \right)^{\alpha_\ell},
\end{align*}
with $\boldsymbol{\eta} = (\eta^{(1)}, \ldots, \eta^{(N)})^T$. Then, define 
\begin{align*}
  \mathcal{S}_k &\coloneqq \mathrm{span} \left\{ \boldsymbol{\eta}^\alpha \;\;\big|\;\; \alpha \in \Upsilon_k\right\}, &  \dim \mathcal{S}_k = |\Upsilon_k|,
\end{align*}
%\akil{Is there a $k$ subscript missing above for the $\mathcal{P}$ term?}
where we take $\Upsilon_k$ %\in \N_0^N$
to be the set of multi-indices corresponding to a degree-$k$ approximation,
\begin{align*}
  \Upsilon_k &\coloneqq \left\{ \alpha \in \N_0^N \;\;\big|\;\; |\alpha| \leq k \right\}, & |\Upsilon_k| =& {N+k \choose N}.
\end{align*}
We will perform a quadratic fit with $k = 2$. We use $\alpha^{(1)}, \ldots, \alpha^{(|\Upsilon_2|)}$, an enumeration of the elements of $\Upsilon_2$,  to form a Vandermonde matrix, $\mathbf{M} \in \mathbb{R}^{m \times |\Upsilon_2|}$ defined as,
\begin{align*}
  \mathbf{M}_{i,v} = \boldsymbol{\eta}_i^{\alpha^{(v)}}
\end{align*}
Now, a least-squares fit to the data $\boldsymbol \psi$ is the emulator,
\begin{align*}
  \widetilde{\psi}(\boldsymbol{\eta}) &= \sum_{v=1}^{|\Upsilon_2|} \widehat{\tau}_v \boldsymbol{\eta}^{\alpha^{(v)}}, & \boldsymbol{\widehat{\tau}} &= (\widehat{\tau}_1, \ldots, \widehat{\tau}_{|\Upsilon_2|})^T,
\end{align*}
%\akil{Meeting: Isn't $\tau$ the same as $\widetilde{\psi}$?}
where $\widehat{\boldsymbol{\tau}}$ is given as the least-squares solution to the linear system,
\begin{align*}
  \mathbf{M} \boldsymbol{\widehat{\tau}} = \boldsymbol{\psi}.
\end{align*}
%\akil{Meeting: Some of the $\tau$'s below aren't bolded correctly.}
Once the coefficients $\widehat{\boldsymbol{\tau}}$ are computed, we translate $\widetilde{\psi}$ into the symmetric quadratic form in \cref{Logapprox} using the following identification of the entries of $c$, $\mathbf{d}$ and $\mathbf{A}$, % from \cref{Logapprox}:
\begin{align*}
  c &= \boldsymbol{\widehat{\tau}}_{\mathcal{I}(0)}, & \mathbf{d}_\ell &= \boldsymbol{\widehat{\tau}}_{\mathcal{I}(\mathbf{e}_\ell)}, & \mathbf{A}_{\ell,u} &= \frac{1}{2 - \delta_{\ell,u}} \boldsymbol{\widehat{\tau}}_{\mathcal{I}(\mathbf{e}_\ell + \mathbf{e}_u)}
\end{align*}
%\akil{Meeting: You usage of $i$ here is not the same as when you introduced $\psi_i$.}
where $\delta_{\ell,u}$ is the Kronecker delta, $\mathbf{e}_\ell \in \N_0^N$ is the cardinal unit vector in direction $\ell$ with entry 1 in location $\ell$ and zeros elsewhere, and $\mathcal{I}(\alpha)$ is a function that returns the linear index in $\Upsilon_2$ associated to $\alpha$,
\begin{align*}
  v = \mathcal{I}(\alpha) \hskip 10pt \Longrightarrow \hskip 10pt \alpha = \alpha^{(v)}.
\end{align*}

In order to associate this quadratic fit with a normal distribution, the matrix $\mathbf{A}$ must be positive-definite.
We will explain in \cref{AltMethodSection} how we address situations when $\mathbf{A}$ is not positive-definite. 
However, when $\mathbf{A}$ is positive-definite, we have the following immediate identification of a density $\mathcal{N}$ from this quadratic fit to the log-integrand:
\begin{proposition}\label{prop:laplace-approx}
  If $\mathbf{A}$ in \cref{Logapprox} is positive-definite, then 
  \begin{align}
  \label{eqn:Fit}
  \exp(-{\widetilde{\psi}}(\boldsymbol \eta)) = C \exp\left( - \left(\boldsymbol{\eta} - \boldsymbol{ \mu}\right)^T \boldsymbol{\Sigma}^{-1} \left(\boldsymbol{\eta} - \boldsymbol{\mu}\right)\right),
  \end{align}
  where
  \begin{align}
  \label{eq:laplace-mu-sigma}
    \boldsymbol{{\mu}} = -\frac{1}{2} \mathbf{U} \boldsymbol{\Lambda}^{-1} \mathbf{d} \hspace{1cm} \boldsymbol{{\Sigma}}^{-1} = \mathbf{A} \hspace{1cm} C = \exp(-c + \frac{1}{4} \mathbf{d}^T \boldsymbol{\Lambda}^{-1} \mathbf{d}))
  \end{align}
\end{proposition}
\begin{proof}
Since $\boldsymbol{A}$ is symmetric (by construction) and positive-definite, it has an orthogonal diagonalization
\begin{equation*}
  \mathbf{A} = \mathbf{U} \boldsymbol{\Lambda} \mathbf{U}^T, \hspace{1cm} \mathbf{U} \mathbf{U}^T = \mathbf{I}, \hspace{1cm} \boldsymbol{\Lambda} = \mathrm{diag}(\lambda_1, \ldots, \lambda_N)
\end{equation*}
with positive eigenvalues $\lambda_\ell > 0$ for all $\ell$. Define $\boldsymbol{\gamma} \coloneqq \mathbf{U}^T \boldsymbol{\eta}$, then
\begin{equation*}
  \widetilde{{{\psi}}}(\boldsymbol{\eta}) = c + \mathbf{d}^T \boldsymbol{\gamma} + (\boldsymbol{\gamma})^T \boldsymbol{\Lambda} \boldsymbol{\gamma}.
\end{equation*}
Let $d^{(\ell)}$ and $\gamma^{(\ell)}$ be the components of $\mathbf{d}$ and $\boldsymbol{\gamma}$, then a rearrangement yields
\begin{align*}
 {\widetilde{\psi}(\boldsymbol{\eta})} &= c + \sum_{\ell=1}^N \left(d^{(\ell)} \gamma^{(\ell)} + \lambda_\ell (\gamma^{(\ell)})^2 \right) \\
            &= c - \sum_{\ell=1}^N \frac{(d^{(\ell)})^2}{4 \lambda_\ell} + \sum_{\ell=1}^N \left(\sqrt{\lambda_\ell} \boldsymbol{\gamma}^{(\ell)} + \frac{d^{(\ell)}}{2\sqrt{\lambda_\ell}} \right)^2\\
            &= c - \frac{1}{4} \mathbf{d}^T \boldsymbol{\Lambda}^{-1} \mathbf{d} + \left({\gamma} + \frac{1}{2} \boldsymbol{\Lambda}^{-1} \mathbf{d} \right)^T \boldsymbol{\Lambda} \left( \boldsymbol{\gamma} + \frac{1}{2} \boldsymbol{\Lambda}^{-1} \mathbf{d} \right) \\
            &= c - \frac{1}{4} \mathbf{d}^T \boldsymbol{\Lambda}^{-1} \mathbf{d} + \left( \boldsymbol{\eta} + \frac{1}{2} \mathbf{U} \boldsymbol{\Lambda}^{-1} \mathbf{d} \right)^T \mathbf{A} \left( \boldsymbol{\eta} + \frac{1}{2} \mathbf{U} \boldsymbol{\Lambda}^{-1} \mathbf{d} \right),
\end{align*}
and we achieve \cref{eqn:Fit}, with ${\boldsymbol{\mu}}$, ${\boldsymbol{\Sigma}}$, and $C$ as given in \cref{eq:laplace-mu-sigma}.
\end{proof}

Using the Laplace approximation, we specify the weight function $\mathcal{N}$ using $\boldsymbol{\mu}$ and $\boldsymbol{\Sigma}$ from \cref{eq:laplace-mu-sigma} so that
\begin{align}
    \label{quadFitIntegral}
    \widetilde{p}(\mathbf{y}_j,t_{n+1})&=
    \int_{\mathbb{R}^N} r(\mathbf{y}) \mathcal{N}(\mathbf{y}; {\boldsymbol{\mu}}, {\boldsymbol{\Sigma}}) d{\mathbf{y}}, &
    r(\mathbf{y}) &= \frac{G(\mathbf{y}_j,\mathbf{y})\widetilde{p}(\mathbf{y},t_n)}{\mathcal{N}(\mathbf{y}; \boldsymbol{{\mu}},\boldsymbol{{\Sigma})}}.
\end{align}
We can approximate this integral using a quadrature rule like in \cref{eqn:QuadRuleEta} with quadrature nodes $\{\boldsymbol{\eta}_1, \dots, \boldsymbol{\eta}_m\}$.

\subsubsection{Alternative Method}
\label{AltMethodSection}
In some situations we cannot use the above Laplace approximation procedure. For example, \Cref{prop:laplace-approx} requires that $\mathbf{A}$ be positive-definite, which may occasionally not occur in practice. When $\mathbf{A}$ is not positive-definite, we use the alternative method shown in \cref{AltIntegral}. Additionally, we use the alternative method in the case that the quadrature nodes selected via the procedure described in \cref{LejaPointsSection} make the quadrature rule described in \cref{ssec:iquad} ill-conditioned. 
In practice, when the alternative method is used, it is typically to update the density of mesh points at or near the mesh boundary.

\hl{For the alternative method, we use a temporary set of quadrature nodes denoted as $\{\boldsymbol{\eta}^*_1, \dots, \boldsymbol{\eta}^*_m\}$. Note, the quadrature nodes for the alternative procedure are such that $\{\boldsymbol{\eta}^*_1, \dots, \boldsymbol{\eta}^*_m\}\not \subseteq \{\mathbf{y}_1, \dots, \mathbf{y}_s\}$. 
%\akil{Is $\not\in$ supposed to be $\not\subseteq$?}
This is distinctly different than the quadrature nodes $\{\boldsymbol{\eta}_1, \dots, \boldsymbol{\eta}_m\}$ used in the standard procedure which are members of the global mesh. In the case of the alternative procedure only, the corresponding PDF values, $\widehat{p}(\boldsymbol{\eta}^*, t_{n+1})$, are determined via interpolation and/or extrapolation to avoid ill-conditioned quadrature. In the case of extrapolation, we assign the density value to be $\min_i \widehat{p}(\mathbf{y}_i,t_n)$.} Use of the alternative method varies, but it usually is only used for around 0-2\% of mesh points per time step on average in two-dimensional problems.  

For the alternative procedure, we take advantage of the structure of $G(\mathbf{x},\mathbf{y})$ to procure a weight function so that
\begin{align}
\begin{split}
\label{AltIntegral}
\widetilde{p}(\mathbf{y}_j,t_{n+1})&=
\int_{\mathbb{R}^N}r(\mathbf{y})\mathcal{N}(\mathbf{y}; \mathbf{y}_j+h\mathbf{f}(\mathbf{y}_j), h\mathbf{g}(\mathbf{y}_j)g(\mathbf{y}_j)^T)d\mathbf{y}\\ 
r(\mathbf{y}) &= \frac{G(\mathbf{y}_j,\mathbf{y})\widetilde{p}(\mathbf{y},t_n)}{\mathcal{N}(\mathbf{y}; \mathbf{y}_j+h\mathbf{f}(\mathbf{y}_j),h\mathbf{g}(\mathbf{y}_j)\mathbf{g}(\mathbf{y}_j)^T)}.
\end{split}
\end{align}
For the alternative method, the weight function $\mathcal{N}$ has a mean and variance that depends only on the current point we are updating, $\mathbf{y}_j$. We can approximate this integral using a quadrature rule like in \cref{eqn:QuadRuleEta} with quadrature nodes $\{\boldsymbol{\eta}^*_1, \dots, \boldsymbol{\eta}^*_m\}$.  

\subsection{Quadrature Weights}\label{ssec:iquad}
We now detail how the quadrature rule in \cref{eqn:QuadRuleEta} of
\begin{equation*}
 \int_{\mathbb{R}^N} r(\mathbf{y})\mathcal{N}(\mathbf{y}; \boldsymbol{\mu},\boldsymbol{\Sigma})d\mathbf{y}\approx \sum_{i=1}^m r(\boldsymbol{\eta}_i) \widehat{w}_{i}
      \end{equation*}
      %\akil{Isn't the above an approximation?}
is generated, assuming the quadrature nodes $\{\boldsymbol{\eta}_i\}_{i=1}^m$ are known. \Cref{LejaPointsSection} later describes the selection process of $\{\boldsymbol{\eta}_i\}_{i=1}^m$ from the global mesh. 
%\akil{Why is ``complex'' used here? What are you trying to communicate?}

For each point $\mathbf{y}_j$ in the global mesh, the Laplace approximation of \cref{QuadFit} allows us to write the integral for the update of $\widetilde{p}$ at $\mathbf{y}_j$ as in \cref{eqn:LaplaceFit}. Now, let $\boldsymbol{\Sigma} = \mathbf{L} \mathbf{L}^T$ be any decomposition of $\boldsymbol{\Sigma}$ (e.g., the Cholesky decomposition). Then, integral \eqref{eqn:LaplaceFit} can be rewritten as
\begin{align*}
 \int_{\mathbb{R}^N} r(\mathbf{y}) \mathcal{N}\left(\mathbf{y}; \boldsymbol{\mu}, \boldsymbol{\Sigma}\right) d {\mathbf{y}} \stackrel{\mathbf{y} = \mathbf{L} \boldsymbol{\zeta} + \boldsymbol{\mu}}{=} 
  \int_{\mathbb{R}^N} r(\mathbf{L} \boldsymbol{\zeta} +\boldsymbol{\mu}) \mathcal{N}\left(\boldsymbol{\zeta}; \boldsymbol{0}, \boldsymbol{I}\right) d {\boldsymbol{\zeta}}.
\end{align*}
Under the same map, we define quadrature nodes $\{\boldsymbol{\eta}_i\}_{i=1}^m$, which are in $\mathbf{y}$ space, as
\begin{align}\label{eq:eta-zeta-map}
  \boldsymbol{\eta}_i = \mathbf{L} \boldsymbol{\zeta}_i + \boldsymbol{\mu}
\end{align}
where $\{\boldsymbol{\zeta}_i\}_{i=1}^m$ are nodes in $\boldsymbol\zeta$ space. 
The weights $\widehat{w}_i$ of the quadrature rule in \cref{eqn:QuadRuleEta} are chosen as the interpolatory weights associated to a particular polynomial space. 
The nodes $\{\boldsymbol{\zeta}_i\}_{i=1}^m$ are chosen in a way that guarantees unisolvence of a polynomial interpolation problem, i.e., given a degree-$k$ polynomial family $\mathcal{P}_k$, we construct a unique polynomial $Q \in \mathcal{P}_k$, so that 
%\akil{What is $\mathcal{P}$ here?}
\begin{align*}
  r\left(\boldsymbol{L}\boldsymbol{\zeta}_i+\boldsymbol{\mu}\right) &= Q(\boldsymbol{\zeta}_i), & i =1, \dots, m.
\end{align*}
More precisely, set $m=m_k$ and let $\{\phi_i\}_{i=1}^{m_k}$ be a basis for the degree-$k$ polynomial space $\mathcal{P}_k$, so that 
\begin{align*}
  Q(\boldsymbol{\zeta}) = \sum_{i=1}^{m_k} \widehat{c}_i \phi_i(\boldsymbol{\zeta}), 
\end{align*}
where $\widehat{\boldsymbol{c}} = (\widehat{c}_1, \ldots, \widehat{c}_{m_k})^T$ solves the linear system,
\begin{equation*}
    \mathbf{V}\boldsymbol{\widehat{c}} = \mathbf{r}, \hspace{1cm} \mathbf{V}_{i,v} = \phi_v(\mathbf{\zeta}_{i}), \hspace{1cm} \mathbf{V}\in \mathbb{R}^{m_k \times m_k}
\end{equation*}
where $\boldsymbol{r} = (r(\boldsymbol{\eta}_1), \ldots, r(\boldsymbol{\eta}_{m_k}))^T$. We generate the quadrature weights from exact integration of $Q$ in place of $r$:
\begin{align}\label{eq:zeta-integral-approximation}
  \int r(\boldsymbol{L}\boldsymbol{\zeta} + \boldsymbol{\mu}) \mathcal{N}\left( \boldsymbol{\zeta}; \boldsymbol{0}, \boldsymbol{I} \right) d {\boldsymbol{\zeta}} \approx
  \int Q(\boldsymbol{\zeta}) \mathcal{N}\left( \boldsymbol{\zeta}; \boldsymbol{0}, \boldsymbol{I} \right) d {\boldsymbol{\zeta}} = \sum_{i=1}^{m_k} r\left(\boldsymbol{\eta}_i\right) \widehat{w}_i,
\end{align}
%\akil{Meeting:2.22}
where $\boldsymbol \eta_i$ are the quadrature nodes and the quadrature weights $\widehat{w}_i$ are given by 
\begin{align}\label{eq:quad-weights}
  \boldsymbol{\widehat{w}} = \left(\widehat{w}_1, \ldots, \widehat{w}_{m_k}\right) &= \boldsymbol{\xi}^T \mathbf{V}^{-1}, & \boldsymbol\xi_i &\coloneqq \int \phi_i(\boldsymbol{\zeta}) \mathcal{N}(\boldsymbol{\zeta}; \boldsymbol{0}, \boldsymbol{I}) d {\boldsymbol{\zeta}}.
\end{align}
The expression for $\boldsymbol{\widehat{w}}$ can be somewhat simplified computationally if we choose the basis $\phi_i$ as a family of polynomials that are $L^2$-orthogonal under the weight function $\mathcal{N}$. Since $\mathcal{N}$ is a Gaussian, the appropriate orthonormal polynomial family are (normalized and tensorized) Hermite polynomials. In particular, let $\alpha^{(1)}, \ldots, \alpha^{(m_k)}$ be an(y) enumeration that satisfies $|\alpha^{(i)}| \leq |\alpha^{(i+1)}|$, then we consider the basis
\begin{align}\label{eq:phi-multid}
  \phi_i(\boldsymbol{\zeta}) = \prod_{\ell=1}^N \widehat{h}_{\alpha^{(i)}_\ell}\left(\zeta_\ell\right),
\end{align}
where $\widehat{h}_i(\cdot)$ is the degree-$i$ normalized univariate Hermite polynomial, satisfying the orthogonality condition,
\begin{align*}
  \int_\mathbb{R} \widehat{h}_i(\zeta) \widehat{h}_v(\zeta) \mathcal{N}\left(\zeta; 0, 1\right) d \zeta &= \delta_{i,v}, & \deg \widehat{h}_i &= i,
\end{align*}
and $\delta_{i,\ell}$ is the Kronecker delta. The uniqueness of each $\widehat{h}_i$ is assured if we insist that the leading coefficient is positive. Since $\mathcal{N}(\cdot; 0, 1)$ is a probability density, $\widehat{h}_0(\zeta) \equiv 1$. The chosen basis defined from \cref{eq:phi-multid} for $\mathcal{P}_k$ then satisfies the following multivariate orthogonality condition,
%\akil{Is $P$ supposed to be $\mathcal{P}$?}
\begin{align}\label{eq:hermite-orthonormality}
  \int_{\mathbb{R}^N} \phi_i(\boldsymbol{\zeta}) \phi_v(\boldsymbol{\zeta}) \mathcal{N}\left(\boldsymbol{\zeta}; \boldsymbol{0}, \boldsymbol{I}\right) d \boldsymbol{\zeta} &= \delta_{i,v}, & 
  \phi_1(\boldsymbol{\zeta}) &\equiv 1,
\end{align}
so that the moments $\boldsymbol{\xi}$ in \cref{eq:quad-weights} are given by $\boldsymbol\xi_i = \delta_{i,1}$. Therefore, with this basis, \cref{eq:quad-weights} implies that the quadrature weights are simply given as the first row of $\mathbf{V}^{-1}$,
\begin{align}\label{eq:quad-w}
  \boldsymbol{\widehat{w}}^T = \left(\mathbf{V}^{-1}\right)_{1,:}.
\end{align}
In terms of the quadrature nodes $\{\boldsymbol{\eta}_i\}_{i=1}^{m_k}$, we connect $\omega_i$ from \cref{eqn:DiscreteC-K} to $\widehat{w}_i$ as
$$\omega_i = \frac{r(\boldsymbol{\eta}_i)}{G(\mathbf{y}_j, \boldsymbol{\eta}_i) \widehat{p}(\boldsymbol{\eta}_i, t_{n})}\widehat{w}_i.$$

\subsection{(Weighted) Leja Sequences}
\label{LejaPointsSection}
We now describe the selection of quadrature nodes $\{\boldsymbol{\eta}_i\}_{i=1}^{m_k}$. Fix a global mesh index $j$, and let ${\mathfrak{Z}}$ denote the set of candidate Leja points determined from the $(\mathbf{L}, \boldsymbol{\mu})$-affine map found via the Laplace approximation in \cref{QuadFit} where $\boldsymbol{\Sigma} = \mathbf{L} \mathbf{L}^T$. More specifically, the global mesh is transformed so that
\begin{equation*}
    \boldsymbol{\zeta}_i \coloneqq \mathbf{L}^{-1}\left(\mathbf{y}_i - \boldsymbol{\mu}\right), \hspace{1cm} \boldsymbol\zeta_i \in \mathfrak{Z}.
\end{equation*}
Note that in practice, we take $\mathfrak{Z}$ as a a subset of the transformed global mesh for computational efficiency.

Our goal is to identify a subset of nodes $\{\boldsymbol{\zeta}_i\}_{i=1}^{m_k}$ , from $\mathfrak{Z}$, which define $\{\boldsymbol{\eta}_i\}_{i=1}^{m_k}$ through \cref{eq:eta-zeta-map}. Recall, $\{\boldsymbol{\eta}_i\}_{i=1}^{m_k}$ are used as quadrature nodes in the approximation given by \cref{eqn:QuadRuleEta}. The nodes $\{\boldsymbol{\zeta}_i\}_{i=1}^{m_k}$ are computed as a discrete weighted Leja sequence from $\mathfrak{Z}$. 
%\akil{You $\eta$'s and $\zeta$'s above don't have subscripts.}

We will formally define one-dimensional Leja sequences, first on a compact interval $[a,b]$, and then weighted on $\mathbb{R}$. Afterwards, we will detail a linear algebra procedure to define $N$-dimensional Leja sequences. 

On an interval $[a,b]$, an unweighted Leja sequence is classically defined as any sequence of points $z_\ell \in [a,b]\subset \mathbb{R}$ for $\ell=1,2,\dots $ that solves the sequential optimization problem,
\begin{equation}
    \label{MaxLeja}
     z_{J+1} = \argmax_{z \in [a,b]}\prod_{\ell=0}^J |z-z_\ell|,
\end{equation}
where $z_0$ is arbitrarily chosen in the interval $[a,b]$ \cite{LejaClassicalEdrei, Leja1957}. Leja sequences are not unique due to the choice of the initial point $z_0$ as well as the potential for multiple maximizers at each step of \cref{MaxLeja}.

In the one-dimensional setting, we are interested in identifying nodes whose corresponding interpolatory quadrature rule is an accurate approximation of the form in \cref{eq:zeta-integral-approximation}. To accomplish this, we must consider integrals over the entire real line with respect to a weight function. For now, we will use a general weight function $W$ which we will later specify as $\mathcal{N}$. 
%\akil{Meeting: What is $W$?}
A naive extension of the optimization in \cref{MaxLeja} that replaces $[a,b]$ by $\mathbb{R}$ is not well defined, so we adopt the strategy from \cite{Narayan2014} that uses weighted Leja sequences and shows that these sequences result in accurate interpolatory quadrature rules with respect to a weight function. 

Let $z_\ell$ for $\ell = 1, 2, \ldots$ be any sequence of solutions to a modified version of \cref{MaxLeja},
    \begin{equation}
      z_{J+1} = \argmax_{z \in \mathbb{R}}\sqrt{W(z)}\prod_{\ell=0}^J |z-z_\ell|
        \label{eqn:LPOpt},
    \end{equation}
where $z_0$ is chosen arbitrarily as an initial point. The use of the weight function penalizes the selection of points at infinity. The use of $\sqrt{W}$ is motivated by the fact that weighted Leja sequences defined by \cref{eqn:LPOpt} satisfy the asymptotic Fekete property, and asymptotically distribute like $W$-weighted Gauss quadrature nodes \cite{Narayan2014}. Empirically, these sequences also form stable quadrature rules for approximating $W$-weighted integrals.

We cannot directly use \cref{eqn:LPOpt} in the DTQ framework because we do not have the freedom to choose points arbitrarily in $\mathbb{R}$. Instead, we pose an optimization problem over the discrete candidate set $\mathfrak{Z}$. We therefore construct the following discrete, $\mathcal{N}$-weighted Leja sequence
    \begin{equation}
      \zeta_{J+1} = \argmax_{\zeta \in \mathfrak{Z}}\sqrt{\mathcal{N}(\zeta;0,1)}\prod_{\ell=0}^J |\zeta-\zeta_\ell|
        \label{eq:leja-discrete}.
    \end{equation}
Ideally, the candidate set $\mathfrak{Z}$ should form a so-called weakly admissible mesh so that the points sufficiently cover the domain of interest \cite{bos_geometric_2011,bos_weakly_2011,xu_randomized_2021}. 

The above discussion holds for one dimension, but the objective function being maximized in \cref{eq:leja-discrete} does not directly generalize to higher dimensions. We will now discuss the one-dimensional problem in a form that can be extended to the multivariate case. The calculation of weighted discrete Leja sequences in \cref{eq:leja-discrete} can be simplified. It is possible to show that \cref{eq:leja-discrete} is equivalent to constructing a Vandermonde-like matrix via a particular kind of greedy determinant maximization \cite{Bos2010} which reduces the process into a simple numerical linear algebra problem. The sequence in \cref{eq:leja-discrete} can be computed from the pivots of a row-pivoted LU factorization of a Vandermonde-like matrix. In particular, let $\widetilde{\mathbf{V}} \in \mathbb{R}^{|\mathfrak{Z}| \times m_k}$
%\akil{Meeting: You should use $m_k$ instead of $m$. Using $m$ is confusing. You refer to the number of quadrature nodes as if they have already been selected, but they have not. Instead $m_k$ is well-defined.}
denote a weighted Vandermonde-like matrix on the candidate points $\mathfrak{Z}$,
\begin{align*}
  \widetilde{\mathbf{V}}_{\ell,i} &= \sqrt{\mathcal{N}(\zeta_\ell;0,1)} ~\widehat{h}_i(\zeta_\ell),\hspace{1cm} \zeta_\ell \in \mathfrak{Z},
\end{align*}
where %$\{\widehat{h}_\ell \}_{\ell=1}^{m_k}$ is any ordered basis satisfying $\mathcal{P}_k = \mathrm{span}\{ \phi_\ell, \;\ell \in [k]\}$. W
we choose $\{\widehat{h}_i \}_{i=1}^{m_k}$ as the univariate, $\mathcal{N}(\cdot;0,1)$-weighted orthonormal Hermite polynomials. With $ \widetilde{\mathbf{P}} \widetilde{\mathbf{V}} = \widetilde{\mathbf{L}} \widetilde{\mathbf{U}}$ the pivoted LU decomposition of $\widetilde{\mathbf{V}}$, a solution to \cref{eq:leja-discrete} %for $1 \leq i \leq m$ 
is given by the first $m_k$ $ \widetilde{\mathbf{P}}$-permuted points
\begin{equation}\label{eq:zeta-pivots}
  \{\zeta_{\ell_1}, \dots, \zeta_{\ell_{m_k}}\} \in \mathfrak{Z}, \hspace{1cm} (\ell_1, \ldots, \ell_{|\mathfrak{Z}|})^T =  \widetilde{\mathbf{P}} \left(1, 2, \ldots, |\mathfrak{Z}| \right)^T.
\end{equation}
Due to the equivalence between the solution to \cref{eq:leja-discrete} and the linear algebra procedure, in practice we compute weighted Leja sequences via an LU factorization with partial row pivoting of a Vandermonde-like matrix \cite{Bos2010}.

\hl{This linear algebra procedure is directly generalizable to multiple dimensions. For $N > 1$, we use $N$-dimensional points $\boldsymbol{\zeta}_i$ to define}
\begin{align*}\label{eqn:NDVtilde}
  \widetilde{\mathbf{V}}_{\ell,i} &= \sqrt{\mathcal{N}(\boldsymbol{\zeta}_\ell;\boldsymbol{0},\boldsymbol{I})} ~\phi_i(\boldsymbol{\zeta}_\ell),\hspace{1cm} \boldsymbol{\zeta}_\ell \in \mathfrak{Z},
\end{align*}
\hl{where we use the multivariate, $\mathcal{N}(\cdot;\boldsymbol{0},\boldsymbol{I})$-weighted orthonormal Hermite polynomial basis $\{\phi_{i}\}_{i=1}^{m_k}$. Then, we use the same greedy determinant maximization procedure (implemented via the pivoted LU approach) used in the one-dimensional case to compute the $N$-dimensional Leja sequence. }

We emphasize that the points selected from the pivots of the LU factorization make up a weighted discrete Leja sequence and are used as quadrature nodes in the approximation \eqref{eq:zeta-integral-approximation}. %which we emphasize are computable by a simple LU decomposition on a weighted Vandermonde-like matrix on $\mathfrak{Z}$ for $N$-dimensional problems.
For the overall DTQ method, the procedure above must be repeated for every point in the global mesh (i.e., for every global index $j$). Our computation of these discrete weighted Leja sequences makes use of code from the PyApprox package \cite{PyApprox}. %https://sandialabs.github.io/pyapprox/adaptive_leja_sequences.html

\subsection{Quadrature Rule Condition
Number: Leja Point Reuse\\ and Alternative Method Use}
Although the Laplace-approximated affine map parameters $(\boldsymbol{L}, \boldsymbol{\mu})$ are recomputed at every time step, to save computational effort, we recompute Leja sequences only if a stability condition is violated. Leja points are attempted to be reused from time step to time step as long as the condition number
\begin{equation*}
\label{condNumberGamma}
    \Gamma \coloneqq \left\| \boldsymbol{\widehat{w}}\right\|_1  <1+\epsilon,
\end{equation*}
where $\|\cdot\|_1$ is the $\ell^1$ norm on vectors, $\boldsymbol{\widehat{w}}$ are the interpolatory quadrature weights defined in \cref{eq:quad-w}, and $\epsilon$ is an adjustable threshold. In practice, the number of mesh points for which we can reuse Leja sequences from the previous time step depends on the drift and diffusion. However, we find that we are often able to reuse a large portion of the Leja points from the previous time step, which increases the speed of the algorithm substantially. We will quantify Leja point reuse in \cref{sec:results} for several examples.

\hl{We also use the condition number to check if we should revert to the alternative procedure. More specifically, if our computation results in a quadrature rule with condition number such that $\Gamma > cond_{alt}$, we assume that the quadrature rule will not yield an accurate result, and revert to using the alternative procedure.}

\section{DTQ with an Adaptive Mesh}
In this section, we describe the adaptive part of the procedure, which updates the mesh based on the density as time evolves. We outline the overall adaptive DTQ method in \Cref{AdaptiveDTQ}.

\begin{algorithm}[h!]
\caption{\hl{Adaptive DTQ}}
\label{AdaptiveDTQ}
\begin{algorithmic}[1]
\Procedure{AdaptiveDTQ}{}:
\While{step forward}:
    \State add points to mesh boundary if needed (\cref{AddingBoundaryPoints})
    \State remove points from mesh boundary if needed (\cref{RemovingBoundaryPoints})
    \For{each mesh point}:
        \State attempt local Laplace approximation (\cref{QuadFit})
        \If{the Laplace approximation is successful}:
            \If{attempting to reuse Leja Points}:
            \State compute quadrature weights (\cref{ssec:iquad})
            \State{step forward using \cref{eqn:QuadRuleEta} with $r$ from \eqref{quadFitIntegral}} evaluated at Leja points
                \If{condition number is sufficient}
                \State{save updated density value}  \State{record if we should try to reuse Leja points at next time step} \State{\textbf{continue}}
                \EndIf
            \EndIf
        \EndIf
            
        \State{attempt Leja point computation (\cref{LejaPointsSection}})
        \If{computed Leja points successfully}
        \State compute quadrature weights (\cref{ssec:iquad})
            \State{step forward using \cref{eqn:QuadRuleEta} with $r$ from \eqref{quadFitIntegral}} evaluated at Leja points
            \If{condition number is sufficient}
            \State{save updated density value}
            \State{record if we should try to reuse Leja points at next time step}
        \State{\textbf{continue}}
        \EndIf
        \EndIf
        \State{use alternative method (\cref{AltMethodSection})}
        \State{save updated density value}
    \EndFor
    \EndWhile
    \EndProcedure
  \end{algorithmic}
\end{algorithm}
\noindent The adaptive procedures attempts to reduce the number of mesh points required to compute $\widehat{p}(\mathbf{x},t)$ by adaptively updating the mesh as the solution evolves in time. To do this, we form a mesh that covers most of the support/mass of the solution. \Cref{fig:AdjustingMesh} gives a visual example of how the mesh is updated to track the density.
\begin{figure}[h!]
    \centering
    \includegraphics[scale=0.3]{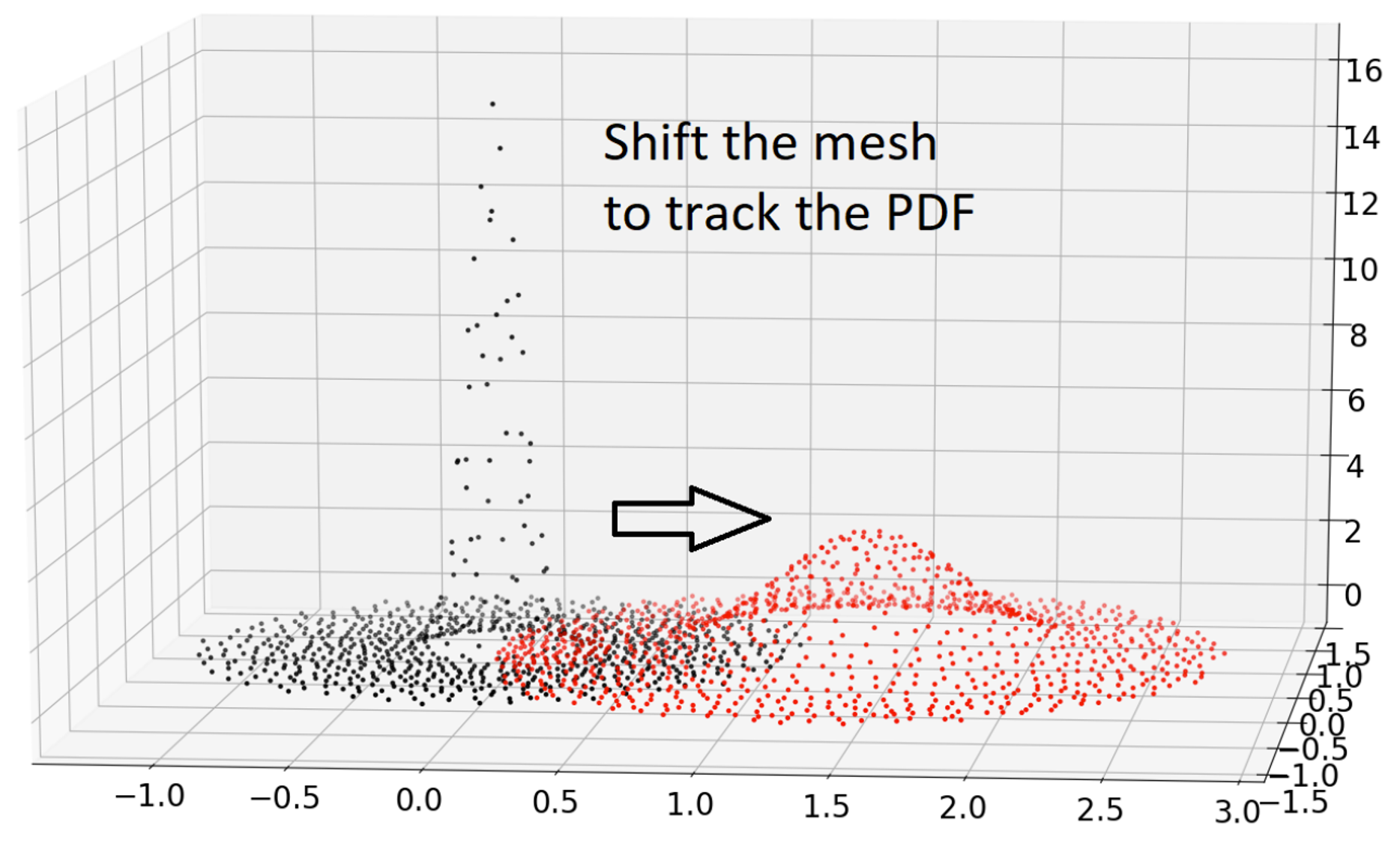}
    \caption{Shifting the mesh to keep track of the density is achieved through adding and removing mesh points.}
    \label{fig:AdjustingMesh}
\end{figure}

\subsection{Identifying the Mesh Boundary}
\label{Delaunay}
Our adaptive procedure adds points to extend the boundary of the mesh. In one-dimensional problems, identifying the boundary points simply involves finding the maximum and minimum mesh coordinate. In order to identify the boundary of the mesh in higher dimensions, we utilize a procedure that combines a mesh triangulation and \textit{alpha shape} procedure.
First, we construct a Delaunay triangulation of the global mesh $\{\mathbf{y}_i\}_{i=1}^s$, which satisfies the condition that no mesh point lies in the interior of any circumscribing sphere of any simplex in the triangulation. The Delaunay triangulation ensures that the minimum angle is maximized for all the triangles in the triangulation to avoid thin, sliver triangles \cite{delaunay1934sphere}. 

The alpha shape algorithm uses the simplices from the Delaunay triangulation to recover the boundary points of a point mesh. A pseudocode explanation of this procedure is given in Algorithm \ref{AlphaShape} and illustrated in \Cref{fig:alphaShapeFigure}. This algorithm works for $N>1$ dimensions. \Cref{fig: AlphaHull} shows a completed identification of boundary points. A survey of alpha shapes is available for more information \cite{edelsbrunner2010alpha}. For the procedure, we select $\widehat{\alpha} = \
(3/2)\Delta{max}$, where $\Delta{max}$ is the enforced maximum distance between points in the mesh.

\begin{algorithm}[h!]
\caption{\hl{Alpha Shape for Finding Boundary Points}}
\label{AlphaShape}
 \begin{algorithmic}[1]
    \Procedure{AlphaShape}{$\widehat{\alpha}$, triangulation}:
    \State compute the radius of the circumsphere for each simplex in the triangulation
    \State simplicesList = all simplices with circumsphere radius $< \widehat{\alpha}$ 
    \State initialize edgesList = []
    \For{simplex in simplicesList}:
        \State add simplex edges to edgesList
    \EndFor
    \State boundaryEdges = edges in edgesList which appear exactly once
    \State boundaryVertices = all vertices associated with edges in boundaryEdges
    \State \Return boundaryVertices
    \EndProcedure
  \end{algorithmic}
\end{algorithm}

\begin{figure}[h!]
    \centering
    \includegraphics[scale=0.35]{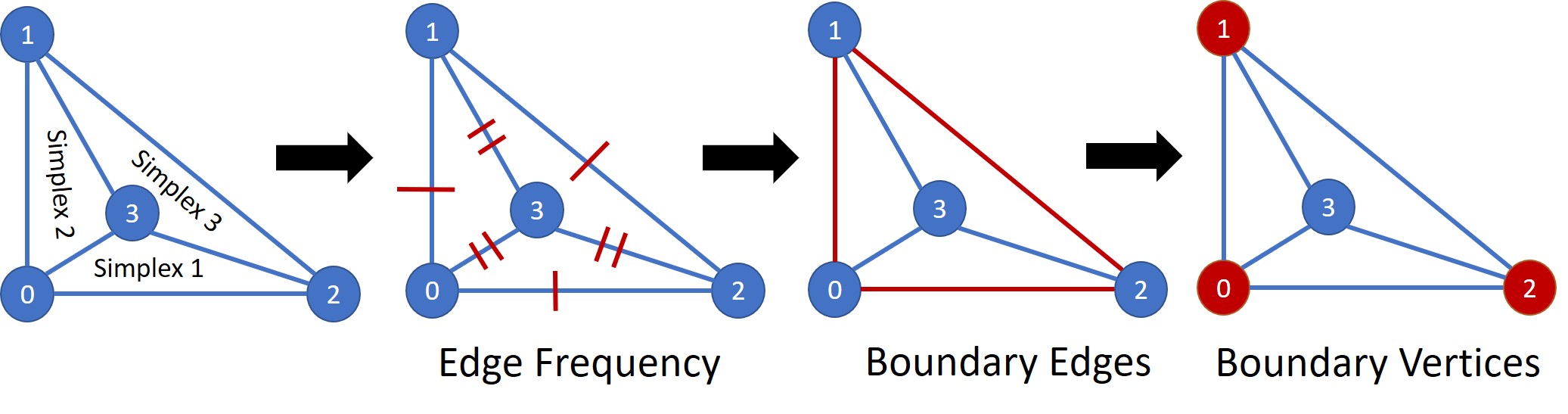}
    \caption{\hl{\textbf{Alpha shape algorithm depiction.} This outlines the alpha shape algorithm. First, identify all edges from the simplices (including duplicates). For example, the edge from point 0 to 3 has frequency two because it appears in both simplex 1 and 2. Then, all edges that appear only once are boundary edges. Finally, the vertices corresponding to the boundary edges are the boundary vertices, in this case, vertices 0, 1, and 2.}}
    \label{fig:alphaShapeFigure}
\end{figure}

\begin{figure}[h!]
    \centering
    \includegraphics[scale=0.2]{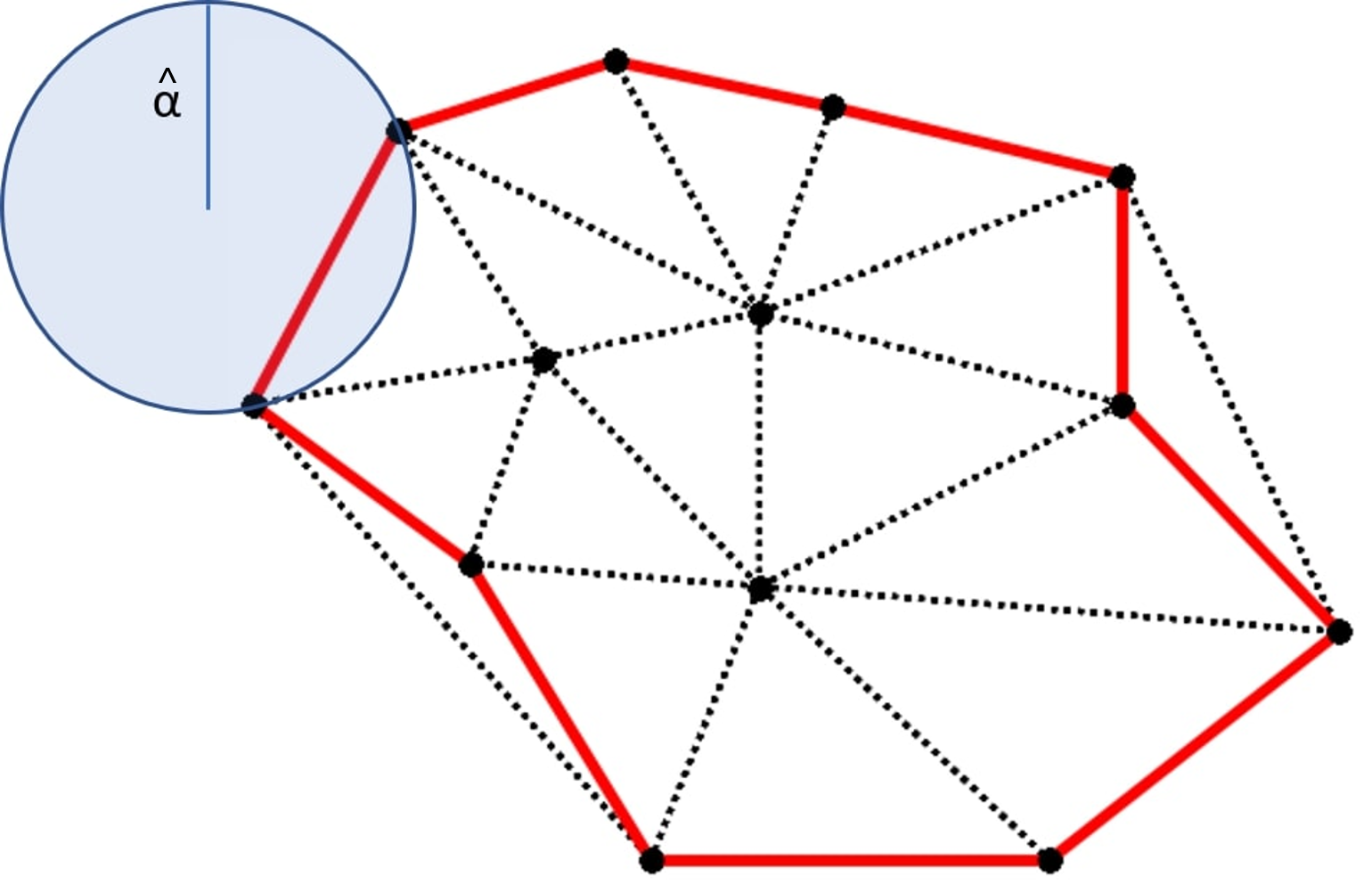}
    \caption{A depiction of a completed two-dimensional alpha shape procedure. The Delaunay triangulation is shown, and the red edges indicate the boundary that is identified.}
    \label{fig: AlphaHull}
\end{figure}

\subsection{Adding Boundary Points}
\label{AddingBoundaryPoints}
After the mesh boundary is identified, as described in the previous section,
we add mesh points based on a prescribed tolerance parameter, $\beta$, which is tuned to maintain a sufficiently small value of $\widehat{p}$ on the mesh boundary. For each boundary point where $\widehat{p}$ is larger than $10^{-\beta}$, \hl{candidate points are generated as equispaced points on an $N$-dimensional grid with spacing $\Delta{max}$}. For a candidate point to be added to the mesh, the distance to the closest point in the mesh must be greater than a set value $\Delta{min}$ and it must be less than a set value $\Delta{max}$. Essentially, $\Delta{max}$ and $\Delta{min}$ are the enforced minimum and maximum distances, respectively, between mesh points. \hl{New mesh coordinates added at time $t_{n+1}$ are assigned a density value of $\min_i \widehat{p}(\mathbf{y}_i, t_n)$.
During the simulation, we begin adding points at time step number $step_{AC}\in \mathbb{Z^+}$, and from time step $step_{AC}$ we run the procedure for adding points every $step_{A}\in \mathbb{Z^+}$ time steps. For example, if $step_{AC} = 1$ and $step_{A}=3$ we would start adding points every three time steps starting at the first time step. We typically take $step_{A}=1$ so that we check to add points every time step after $step_{AC}$.}

\subsection{Removing Points}
\label{RemovingBoundaryPoints}
We remove the mesh points that are deemed unnecessary in terms of the value of the density $\widehat{p}$. Points are removed based on the prescribed tolerance, $\beta$ from \cref{AddingBoundaryPoints}, which is used to maintain a sufficiently small value of $\widehat{p}$ at the boundary. More specifically, all mesh points (anywhere in the mesh, not just at the boundary) that are smaller than $10^{-\beta-0.5}$ are removed from the mesh. This procedure for adding points is typically run periodically; however, not at every time step.

\hl{We first run this removal procedure on time step $step_{RC}\in \mathbb{Z^+}$, and subsequently run the procedure for removing points every $step_{R}\in \mathbb{Z^+}$ time steps. We typically take $step_{RC}, step_{R} \approx 10$ so that we do not start removing points right away nor at every time step.}

\section{Results}
\label{sec:results}
The code for the adaptive DTQ procedure is located on GitHub \cite{DTQ-Github}. In this section, we demonstrate \Cref{AdaptiveDTQ} through several examples. \hl{Some examples use spatially-dependent drift and diffusion functions. Problems from one to five dimensions are included.} \Cref{tab:DTQParameterDescrptions} \hl{outlines parameters used for the Adaptive DTQ along with a brief description. }

%For notational purposes, we will let ${x}^{(1)}, \dots {x}^{(N)}$ denote the spatial dimensions of the problem and we will use $x$ in the $N=1$ case.
\begin{table}[h!]
\caption{\label{tab:DTQParameterDescrptions} \hl{\textbf{DTQ parameter symbols and descriptions}}.}
\centering
\begin{tabular}{c| p{11cm}} 
 Parameter & Description \\ [0.5ex]
 \hline
 $\epsilon$ & $1+\epsilon$ is the condition number threshold for Leja point reuse  \\
 $cond_{alt}$ & The condition number threshold for using the Alternative method  \\
 ${LP}_Q$ & Number of Leja points used for the quadrature rule \\ 
 size of $\mathfrak{N}$ & The size of the set used for the Laplace approximation\\ 
 size of ${\mathfrak{Z}}$ & The number of candidate Leja points\\
$step_{AC}$ & The first step number where adding points is considered\\ 
$step_{RC}$ & The first step number where removing points is considered\\ 
$step_A$ & Determines how often points are added \\ 
$step_{R}$ & Determines how often points are removed \\ 
 $\Delta{min}$ & Minimum distance allowed between points  \\
 $\Delta{max}$ & Maximum distance allowed between points \\
 $\mathcal{R}$ & Used to define the radius of points in the initial mesh \\
  $h$ & The temporal time step size \\
$\beta$ & Used to define values for adding and removing points  \\
 \hline
\end{tabular}
\end{table}

\subsection{Commonalities Among Examples}
We demonstrate our adaptive DTQ method in this section through a variety of examples. Some examples use spatially-dependent drift and diffusion functions. Problems from one to five dimensions are included. We first describe some experimental setup characteristics that are common among all examples.

\subsubsection{Parameter Values}
\Cref{tab:DTQParameterValues}\hl{ outlines some parameters which are consistent among all examples in this section. We note that these values are adjustable and all parameter values given are likely not optimal, but they work well for the examples we show. The size of $\mathfrak{N}$ depends on if Leja quadrature nodes or nearest neighbor points are used} (see \cref{QuadFit}).
\begin{table}
\caption{\label{tab:DTQParameterValues} \hl{\textbf{DTQ parameter values by dimension}. The values of the parameters used in the examples in this manuscript. The two numbers given for the size of $\mathfrak{N}$ depend on if Leja points or nearest neighbors are used} (see \cref{QuadFit}).}
\centering
\begin{tabular}{c|c|c|c|c|c} 
 Parameter & $N=1$ & $N=2$ & $N=3$ & $N=4$ & $N=5$ \\ [0.5ex]
 \hline
 $\epsilon$ & 0.1 & 0.1 & 0.1 &0.1 & 0.1 \\ [0.5ex]
 $cond_{alt}$ & 5 & 5 & 5 &5 & 5 \\ [0.5ex]
   ${LP}_Q$ & 6 & 10 & 15 &15 & 40 \\ [0.5ex]
  size of $\mathfrak{N}$ &  ${LP}_Q$ or 20 &  ${LP}_Q$ or 20 &  ${LP}_Q$ or 150 &  ${LP}_Q$ or 200 &  ${LP}_Q$ or 300 \\ [0.5ex]
 size of ${\mathfrak{Z}}$ & 50 & 150 & 150 & 250 & 450 \\[0.5ex]
  $step_{AC}$ & 1 &1& 1 &1 & 1 \\[0.5ex]
  $step_{RC}$ & 10 & 10 & 10 &10 & 10 \\[0.5ex]
  $step_A$ & 1 & 1 & 1 &1 & 1 \\[0.5ex]
  $step_R$ & 10 & 10 & 10 & 10 & 10 \\[0.5ex]
   \hline
\end{tabular}
\end{table}
\hl{Additional parameter values for 
$\Delta{min}, \Delta{max}, \mathcal{R}, h,$ and $\beta$ are given in each individual example.}

\subsubsection{Initial Condition and Initial Mesh}
\label{sec:IC}
The initial condition in these examples is taken to be a Dirac mass centered at the origin, $\mathbf{0}$. We compute the first time step as $\widehat{p}(\mathbf{y}_j, t_h)=G(\mathbf{y}_j, \mathbf{0})$ where $G$ is defined in \cref{eqn:Gxy}.
%\akil{This is only true (exact) for constant drift and diffusion, right?}

\hl{The initial mesh is an equispaced grid of points which is centered at the origin with mesh spacing $\Delta{min}$ and with a radius within the range $[\mathcal{R}-\ \Delta{min}, \mathcal{R}]$, where $\mathcal{R}$ is a parameter defining the initial mesh radius. This range occurs due to the fact that the mesh spacing may not evenly divide into the radius.} \Cref{fig:initialMesh} \hl{shows initial mesh examples with $\mathcal{R}=2$ in 1, 2, and 3 dimensions.}

\begin{figure}[h!]
    \centering
    \includegraphics[scale=0.5]{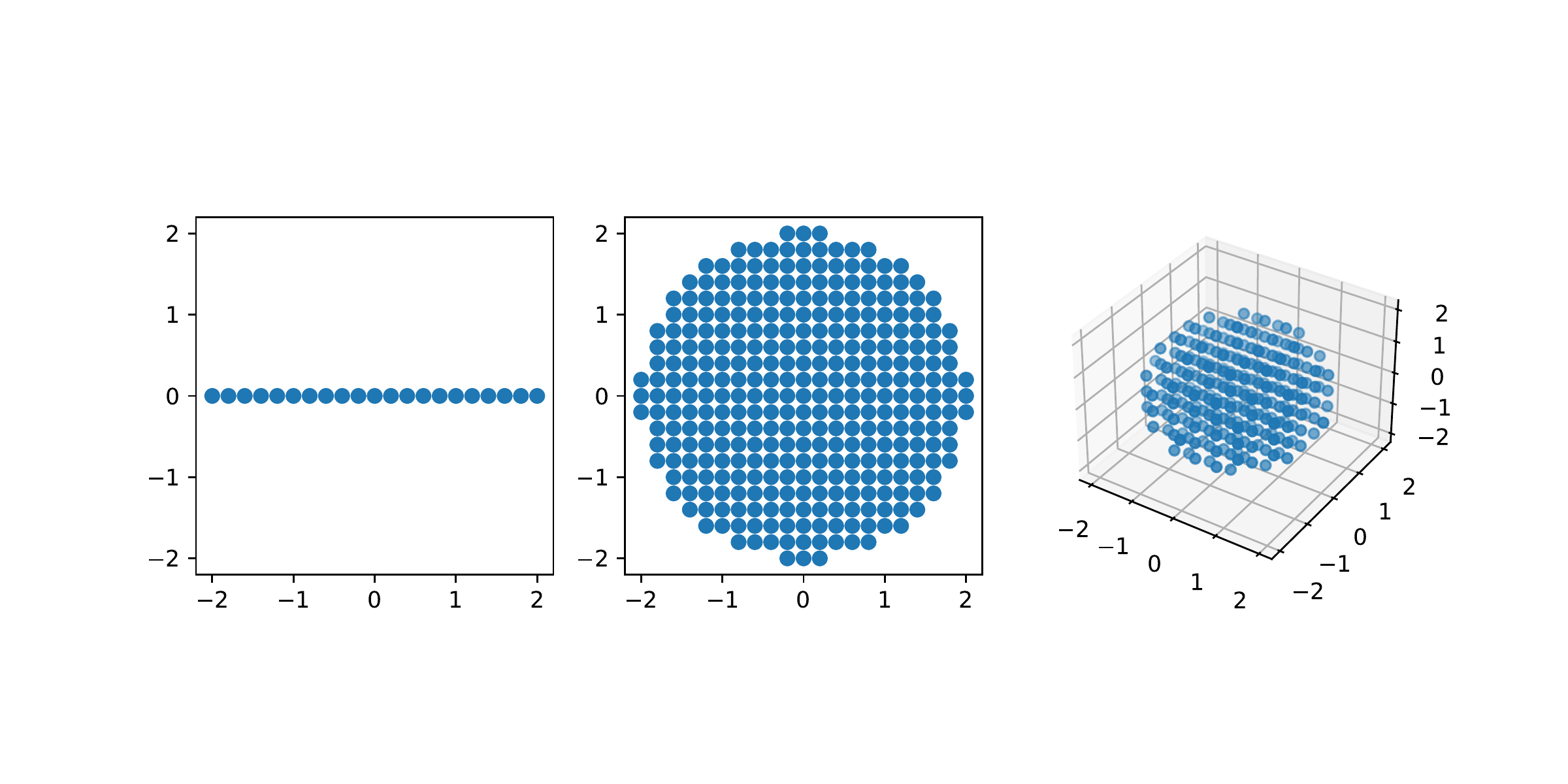}
    \caption{\textbf{\hl{Initial mesh}}. Initial mesh examples with $\mathcal{R}=2$ in $N=1$, $N=2$, and $N=3$ dimensions respectively.}
    \label{fig:initialMesh}
\end{figure}

We report errors in the experiments below, which are measured in the following ways. 
Let $p$ be the exact solution and $\widehat{p}$ be the computed (adaptive DTQ) solution. With $s$ the number of points in the mesh, define the following spatial errors at a fixed time:
\begin{align*}
  L_{2p} &= \sqrt{\frac{1}{\sum_{i=1}^s p(\mathbf{y}_i)}\sum_{i=1}^s \bigg((p(\mathbf{y}_i)-\widehat{p}(\mathbf{y}_i))^2p(\mathbf{y}_i)\bigg)} \\ 
  L_{2} &= \sqrt{\frac{1}{s}\sum_{i=1}^s (p(\mathbf{y}_i)-\widehat{p}(\mathbf{y}_i))^2} \\
  L_{1} &= \frac{1}{s}\sum_{i=1}^s |p(\mathbf{y}_i)-\widehat{p}(\mathbf{y}_i)| \\
  L_{\infty} &= \max_{i \in [s]} |p(\mathbf{y}_i)-\widehat{p}(\mathbf{y}_i)|,
\end{align*}
which are, respectively, approximations of the $L^2_p(\R^N)$, $L^2(\R^N)$, $L^1(\R^N)$, and $L^\infty(\R^N)$ norms.

\subsubsection{Error, Leja Point Reuse, and Alternative Procedure Use Metrics}\label{sec:metrics}
\hl{Furthermore, we report statistics on the average Leja point reuse and the alternative method use. Let $N_{steps}$ be the number of time steps taken in the simulation, then}
\begin{equation*}
\small
\text{Average Leja Reuse \%} = \frac{100}{N_{steps}-2} \sum_{i=3}^{N_{steps}} \frac{\text{\# of points reusing Leja points at time } t_i}{\text{\# of points in the mesh at time } t_i}.
\end{equation*}
\hl{We start at the third time step because Leja points are first computed in the second time step since the first time step is computed directly as discussed in} \cref{sec:IC}). Similarly, 
\begin{equation*}
\small
\text{Average Alt. Method Use \%} = \frac{100}{N_{steps}-1} \sum_{i=2}^{N_{steps}} \frac{\text{\# of points using alt. method at time } t_i}{\text{\# of points in the mesh at time } t_i}.
\end{equation*}
\hl{Again, we do not include the first time step because it is computed directly. }

\subsection{One-Dimensional Example} \label{oneDExample}
%\subsubsection{Constant Drift and Diffusion}
\hl{For this example, we consider the one-dimensional SDE with} \begin{align*}
  \mathbf{f}(x) &= 2, & 
    \mathbf{g}(x) &=1.
\end{align*}
\hl{We set simulation parameters as $(\Delta{min}, \Delta{max}, \mathcal{R}, h, \beta) = (0.4, 0.4, 2, 0.05, 4)$. The average Leja reuse was 94\% and the average alternative method use was 1.8\%. The $L_{2p}$ error at $t=10$ is 3.5e-05. The PDF at several times is shown in} \Cref{fig:1Devo}.

\begin{figure}[h!]
    \centering
    \includegraphics[scale=0.4]{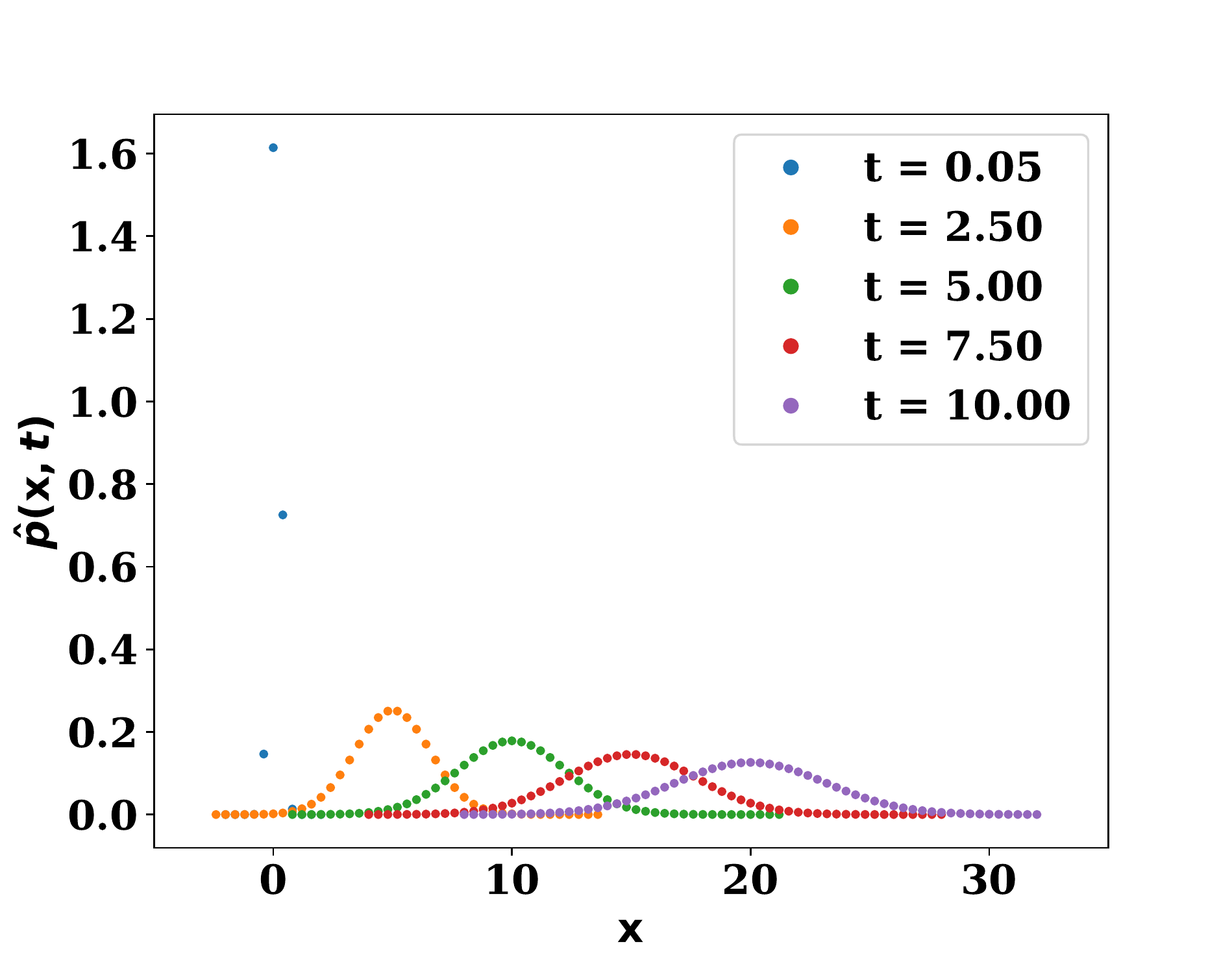}
    \caption{\textbf{\hl{One-dimensional constant drift and diffusion example}}. This figure shows the computed density at several different times indicated in the legend.}
    \label{fig:1Devo}
\end{figure}

\subsection{Two-Dimensional Examples}
Now, we will cover several examples using our adaptive DTQ procedure. 
\subsubsection{Constant Drift and Diffusion}\label{ssec:moving-hill}
Consider the solution to the SDE in \cref{eqn:SDE} with constant drift and diffusion
\begin{align}
\label{eqn:ConstantDriftDiff}
  \mathbf{f}(\mathbf{x}) &= \left(\begin{array}{c} C_1 \\ 0 \end{array}\right), & 
    \mathbf{g}(\mathbf{x}) &= \left(\begin{array}{cc} C_2 & 0 \\ 0 & C_2 \end{array}\right).
\end{align}

\noindent This SDE corresponds to the Fokker-Planck PDE
$$\frac{\partial p(\mathbf{x},t)}{\partial t}=-C_1\frac{\partial}{\partial x}p(\mathbf{x},t)+\frac{C_2^2}{2}\frac{\partial^2}{\partial (x^{(1)})^2}p(\mathbf{x},t)+
\frac{C_2^2}{2}\frac{\partial ^2}{\partial (x^{(2)})^2}p(\mathbf{x},t)$$ 
which has the exact solution  
\begin{equation}
    p(\mathbf{x},t)=\frac{1}{4\pi (C_2^2/2)t}\exp\left(\frac{-((x^{(1)}-C_1t)^2+(x^{(2)})^2)}{4(C_2^2/2)t}\right).
    \label{eqn:ExactConstantSoln}
\end{equation}

Here, we select $C_1=C_2=1$, and we define $(\Delta{min}, \Delta{max}, \mathcal{R}, h) = (0.2, 0.2, 2, 0.01)$. We explore the effect of $\beta$, the tolerance parameter for allowable density values on boundary nodes, from \cref{AddingBoundaryPoints}. The error at time $t = 1.15$ for varying $\beta$ are shown in \Cref{tab:moving-hill}.\\

\begin{table}[h!]
\caption{\textbf{Adaptive DTQ errors for different values of the boundary tolerance parameter $\beta$ at time $t=1.15$ for the moving hill example of \cref{ssec:moving-hill}}. Also shown are the number of points $s$ in the adaptive mesh at $t = 1.15$. Generally, as $\beta$ increases, the error decreases because we use more points to cover a larger area of the domain.}\label{tab:moving-hill}
    \centering
    \begin{threeparttable}
        \begin{tabular}{r c c c c r}
            \toprule
            %\multicolumn{5}{l}{Minimal Working Example}  \\
            \midrule
            \multicolumn{1}{c}{$\beta$} & \mc{$L_{2p}$ Error}& \mc{$L_{2}$ Error}& \mc{$L_1$ Error}& \mc{$L_{\infty}$ Error} & \mc{\# Points $s$}\\
            \midrule
1&1.7e-02&1.6e-02&1.2e-03&3.5e-02&181\\ \hline 2&1.8e-03&1.6e-03&7.5e-05&2.7e-03&556\\ \hline 3&2.1e-04&2.2e-04&5.1e-06&3.3e-04&994\\ \hline 4&1.8e-05&1.9e-05&3.0e-07&3.0e-05&1436\\ \hline 5&1.4e-06&2.1e-06&1.7e-08&3.5e-06&1866\\ \hline 6&8.5e-08&2.1e-07&7.6e-10&4.3e-07&2319\\ \hline 7&8.5e-09&2.5e-08&6.2e-11&4.8e-08&2780\\ \hline 8&5.6e-10&2.6e-09&3.2e-12&8.0e-09&3295\\ \hline 9&2.3e-11&1.8e-10&8.0e-14&5.3e-10&3588\\ \hline 10&1.6e-12&1.5e-11&6.7e-15&4.7e-11&3996\\ \hline 
         \hline 
        \end{tabular}
    \end{threeparttable}
\end{table}

We observe that $\beta$ has a somewhat direct control on error of the approach for this simple example.
For larger $\beta$, more points are used to form a mesh on a larger region, which improves the overall accuracy of the method. \Cref{tab:moving-hill} shows that as $\beta$ increases, the error tends to decrease. Thus, the adaptive DTQ method can be quite accurate if $\beta$ is chosen appropriately. With that said, the selection of $\beta$ must be balanced with the computational cost associated with a larger mesh. The number of mesh points $s$ at the last time step is also shown.

\subsubsection{Adaptive Leja Quadrature vs. Equispaced Trapezoidal DTQ}
\hl{We will now compare and contrast our adaptive DTQ method with the trapezoidal rule DTQ method. For notational purposes, we will refer to the two differing DTQ methods as follows:}
\begin{itemize}
    \item DTQ$_{TR}$: non-adaptive density tracking by quadrature method using an equispaced mesh and a trapezoidal quadrature rule (\cref{ssec:tensor-DTQ}).
    \item DTQ$_{LQ}:$ adaptive density tracking by quadrature method using a interpolatory Leja quadrature rule with the ability to add and remove mesh points.
\end{itemize}

\hl{In many lower-dimensional problems (e.g. $N=1, 2$) over a small domain, DTQ$_{TR}$ will be computationally faster than DTQ$_{LQ}$. With that said, DTQ$_{TR}$ will typically require more memory due to the need to form a dense grid. As the domain grows, even in smaller dimensions, DTQ$_{LQ}$ becomes more efficient than DTQ$_{TR}$, eventually becoming the preferred method. In this section, we formulate a two-dimensional example where DTQ$_{LQ}$ is both computationally faster to run and uses fewer mesh points than DTQ$_{TR}$. As the dimension of the problem increases, we expect that DTQ$_{LQ}$ will typically be the superior option of the two for many problems because DTQ$_{TR}$ will use significantly more points or simply be too memory intensive.}

We compared the computational timing for DTQ$_{TR}$ against DTQ$_{LQ}$ for the constant drift and diffusion example in \cref{eqn:ConstantDriftDiff} with $C_1=1$ and $C_2=0.6$. The simulation parameters for DTQ$_{LQ}$ are $(\Delta{min}, \Delta{max}, \mathcal{R}, h) = (0.38, 0.38, 2, 0.05)$ with  $\beta \in [2.5, 3, 4, 5, 6]$. For DTQ$_{TR}$, we use $h=0.05$ and equispaced spatial step sizes of $\kappa \in [0.25, 0.2, 0.18, 0.15]$. Notice that we varied the accuracy and cost for each of these two methods by changing a parameter. For DTQ$_{LQ}$, we adjusted the boundary cutoff parameter, $\beta$, and for DTQ$_{TR}$ we adjusted the equispaced spatial step size ${\kappa}$.
 
\hl{Since we are considering a tensorized trapezoidal quadrature rule procedure that does not adaptively update the mesh, we need to determine the proper mesh for DTQ$_{TR}$ \emph{a priori} using the mesh from DTQ$_{LQ}$ as a starting point. In this example, we used information about the meshes from the DTQ$_{LQ}$ result with $\beta = 4$ and added a 0\% and 50\% buffer. The buffer percentage is used to create a tensorized mesh with padding. Psuedocode for how this works is given in} \Cref{algo:buffer} and a depiction of this is shown in \Cref{Fig:PaddingBufferMesh}.

\begin{algorithm}[h!]
\caption{\hl{Form tensorized mesh based on DTQ$_{LQ}$ mesh}}
\label{algo:buffer}
 \begin{algorithmic}[1]
    \Procedure{TensorizedMeshBasedOnDTQ$_{LQ}$Mesh}{$DTQ_{LQ}$ mesh, buffer}:
    \State [$min_1, \dots, min_N$] = min($DTQ_{LQ}$ mesh) in each dimension
    \State [$max_1, \dots, max_N$] = max($DTQ_{LQ}$ mesh) in each dimension
    \State [$pad_1, \dots, pad_N$] = buffer/2*[$max_1 - min_1, \dots, max_N-min_N$]
    \State grid = [$min_1 - pad_1, \dots, min_N-pad_N$] $\times$ [$max_1+pad_1, \dots, max_N+pad_N$]
    \State \Return grid
    \EndProcedure
  \end{algorithmic}
\end{algorithm}

\begin{figure}[h!]
    \centering
    \includegraphics[scale=0.35]{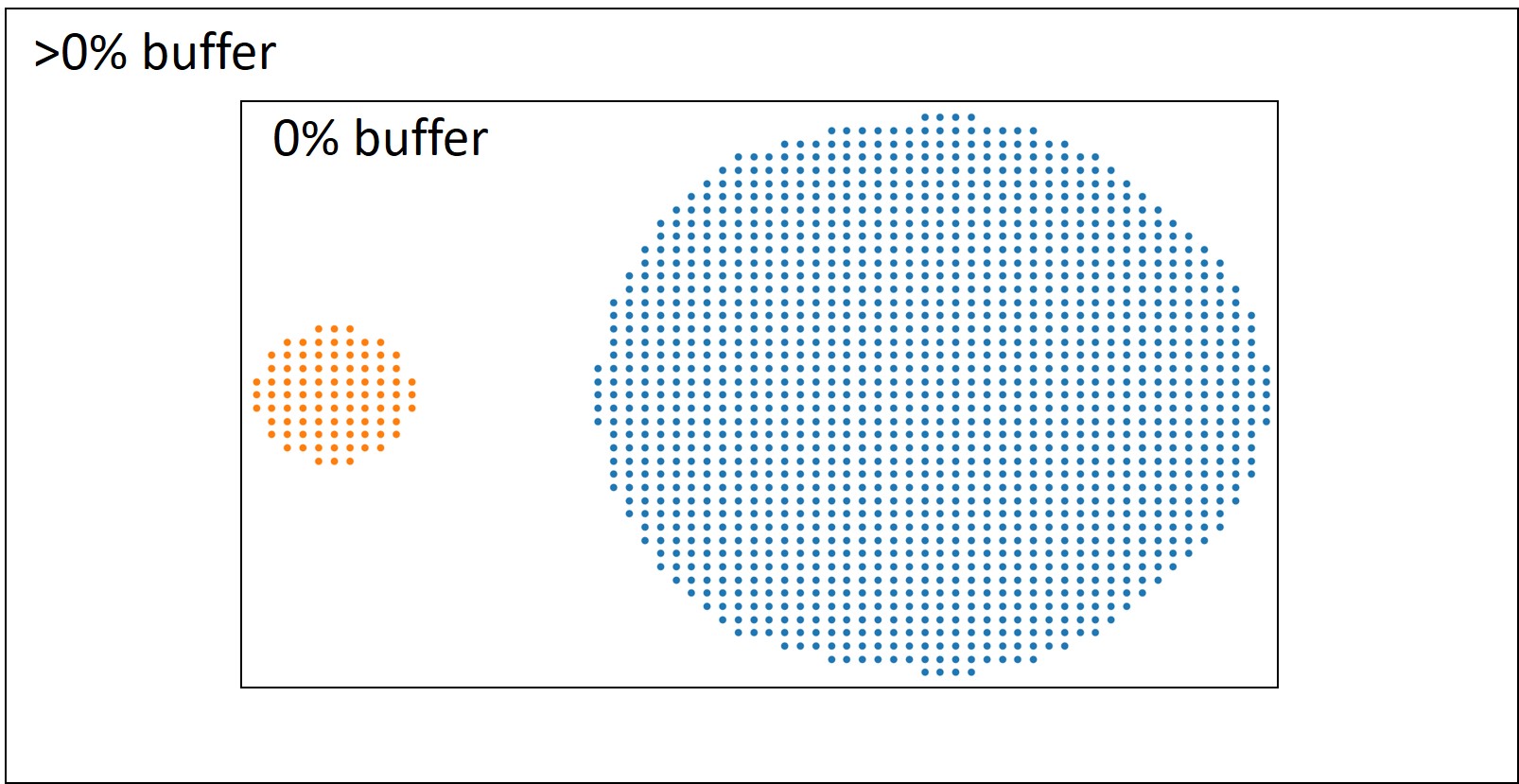}
    \caption{\textbf{\hl{Two-dimensional example of forming a tensorized mesh based on an adaptive mesh with a buffer}. This figure illustrates how we form a tensorized mesh for DTQ$_{TR}$ based on the DTQ$_{LQ}$ mesh with an additional buffer. The orange and blue meshes depict the coverage during a DTQ$_{LQ}$ simulation at the starting and ending time respectively. The rectangular lines indicate the region the equispaced tensorized mesh will fill in given a buffer percentage of $0$ or $>0$. Notice, the padding in each dimension can be different.}}
    \label{Fig:PaddingBufferMesh}
\end{figure}

\hl{Essentially, we form a tensorized mesh for DTQ$_{TR}$ based on the DTQ$_{LQ}$ mesh with an additional buffer which depends on the coverage of the domain during the DTQ$_{LQ}$ simulation. We refer to the $0\%$ buffer as an oracle solution because the knowledge needed to set a mesh which tightly supports the solution in this way requires complete knowledge of the solution support. The $50\%$ buffer still requires knowledge of the solution; however it represents a much looser estimate than the oracle of the required mesh and is meant to be more indicative of an educated guess for the mesh.}

\Cref{fig:T40Time} shows the timings for each simulation recorded relative to the timing of the DTQ$_{LQ}$ simulation with $\beta=2.5$, indicated by the star. The exact $L_{2p}$ errors are compared since the exact solution is known (see \cref{eqn:ExactConstantSoln}). \Cref{fig:T40Time} shows that, for similar effort (relative time), DTQ$_{TR}$ takes about two (0\% buffer) to ten (50\% buffer) times longer than DTQ$_{LQ}$. Additionally, we emphasize again that DTQ$_{TR}$ requires information gathered from the DTQ$_{LQ}$ mesh in order to run the simulation.

\hl{Recall, the enforced mesh spacing for DTQ$_{LQ}$ was kept at 0.38 for all simulations since the $\beta$ parameter was adjusted for accuracy. This mesh spacing was much larger than the mesh spacing of 0.2 and 0.15 used by the accurate DTQ$_{TR}$ simulations. The number of points for each simulation are shown in} \Cref{fig:T40Points}. \hl{We see that DTQ$_{TR}$ required one to two orders of magnitude more points to achieve similar error to DTQ$_{LQ}$. Also, we note that the recorded points shown in} \Cref{fig:T40Points} \hl{are at or near a maximum for the DTQ$_{LQ}$ procedure since they are reported at the final time step. }

An AMD EPYC 7702P 64-Core Processor with 251 GiB of memory was used for this example.

\begin{figure}[h!]
    \centering
    \includegraphics[scale=0.38]{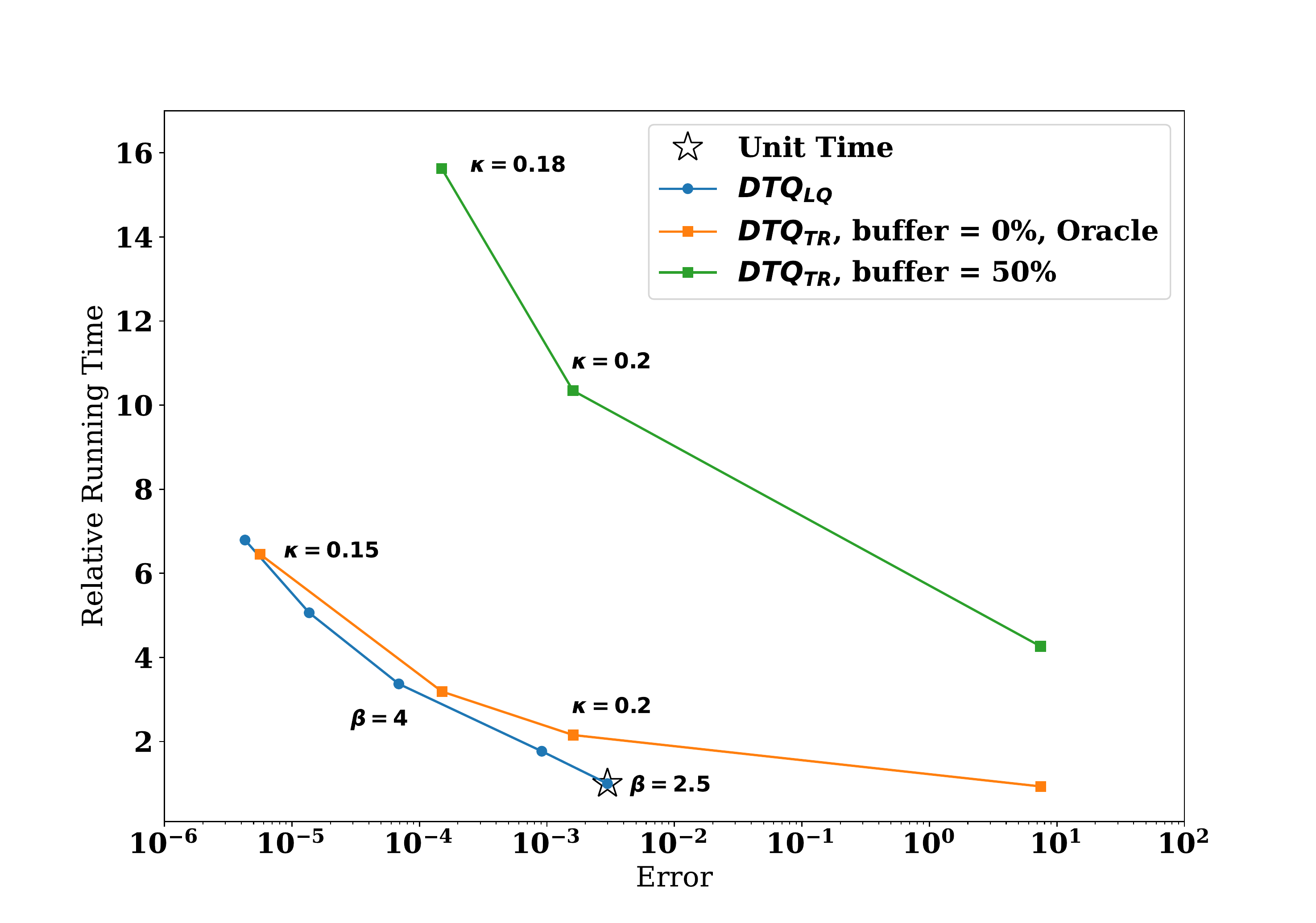}
    \caption{\textbf{DTQ$_{TR}$ and DTQ$_{LQ}$ timing comparison}. In this figure, we compare the error ($L_{2p}$) and timing between DTQ$_{TR}$ and DTQ$_{LQ}$ for a constant drift, constant diffusion example where the end time is $t=40$.}
    \label{fig:T40Time}
\end{figure}

\begin{figure}[h!]
    \centering
    \includegraphics[scale=0.38]{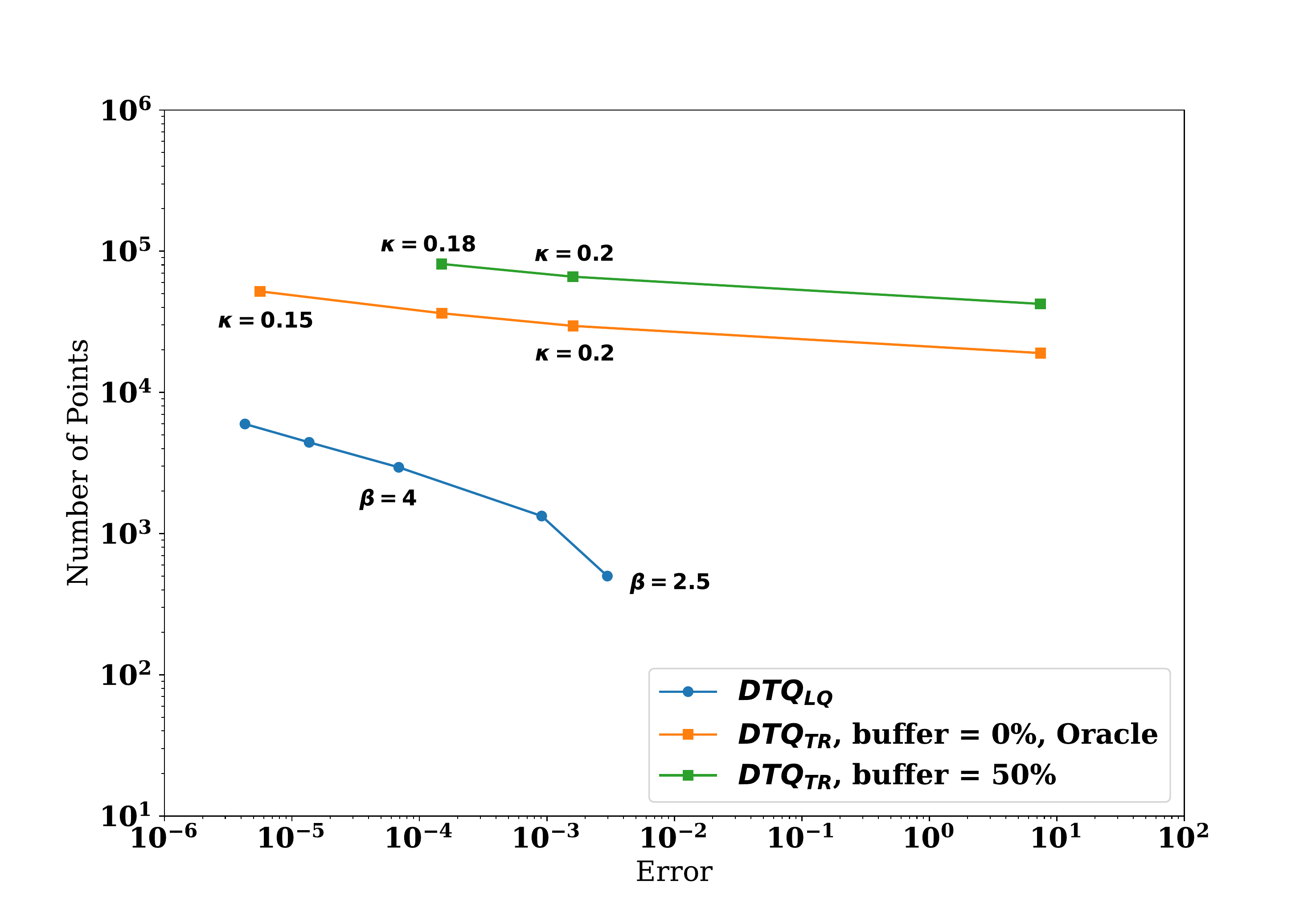}
    \caption{\hl{\textbf{DTQ$_{TR}$ and DTQ$_{LQ}$ end time mesh size comparison}. In this figure, we compare the number of points at the end time $t=40$ between DTQ$_{TR}$ and DTQ$_{LQ}$ for a constant drift, constant diffusion example. The error is given in terms of $L_{2p}$. The trapezoidal rule uses the same number of points per time step whereas the adaptive procedure updates the mesh points to track the density.}}
    \label{fig:T40Points}
\end{figure}

\subsubsection{Erf Drift Example}\label{ssec:four-hills}
We consider the solution to the two-dimensional SDE in \cref{eqn:SDE} with drift and constant diffusion
\begin{align*}
  \mathbf{f}(\mathbf{x}) &= \left(\begin{array}{c} 2\text{erf}(10x^{(1)}) \\ 2\text{erf}(10x^{(2)}) \end{array}\right), & 
    \mathbf{g}(\mathbf{x}) &= \left(\begin{array}{cc} 0.75 & 0 \\ 0 & 0.75 \end{array}\right).
    %\label{example:Erf}.
\end{align*}
The error function used for the drift is defined as 
\begin{equation*}
    \text{erf}(z) = \frac{2}{\sqrt{\pi}}\int_0^z e^{-t^2}dt.
\end{equation*}
In this example involving the erf drift, we use DTQ$_{LQ}$ simulation parameters defined as\\ $(\Delta{min}, \Delta{max}, \mathcal{R}, h, \beta) = (0.25, 0.3, 3, 0.04, 4)$. We also use a finely discretized DTQ$_{TR}$ simulation with parameters of ($\kappa, h) = (0.08, 0.01)$ for comparison. We would have preferred to use $\kappa = 0.05$ or smaller for accuracy purposes; however, running such an experiment was not feasible on available hardware due to memory limitations.

\begin{figure}[h!]
    \centering
    \includegraphics[scale=0.55]{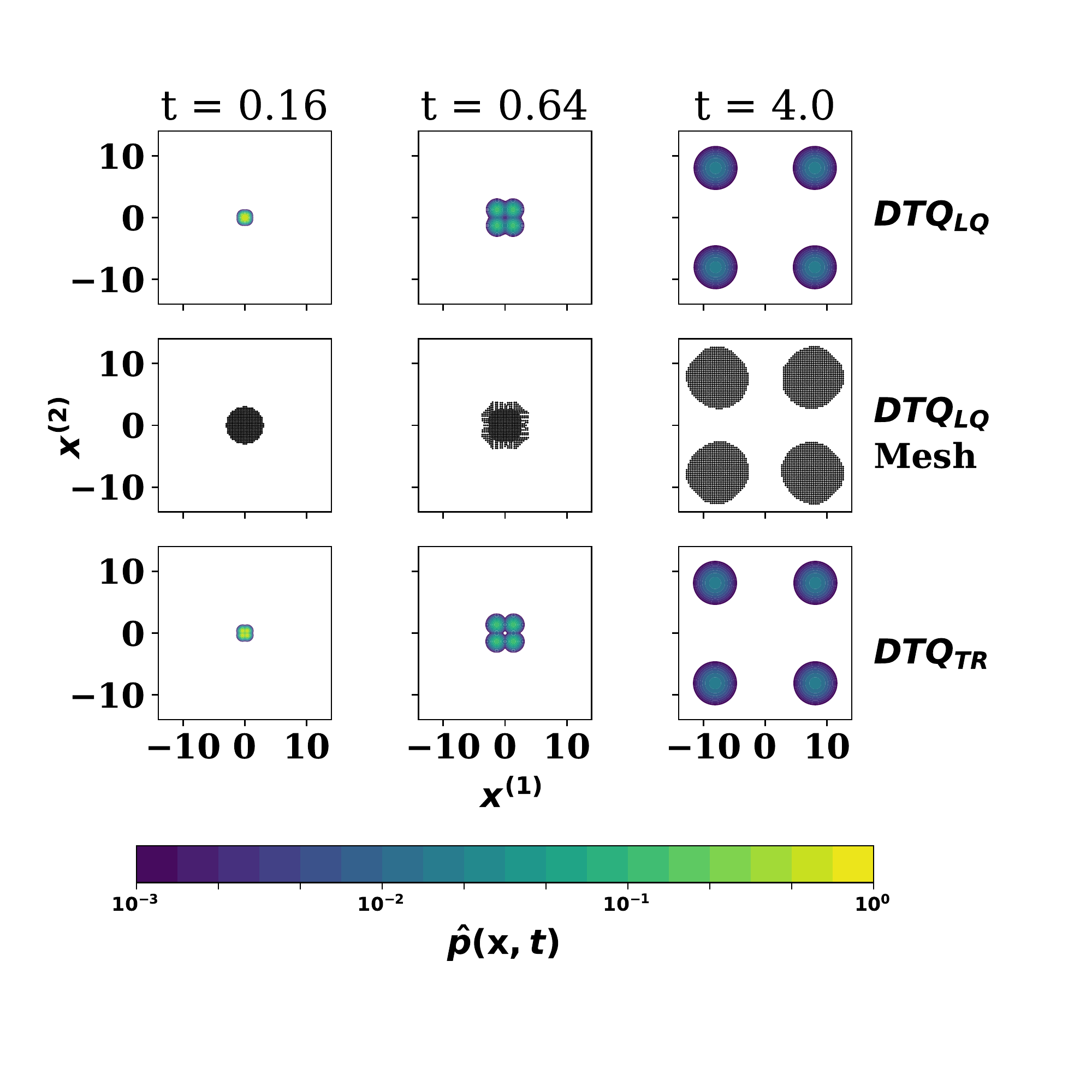}
    \caption{\textbf{Erf drift example.} The solution at the indicated times for the erf drift example of \cref{ssec:four-hills}. Top row: Density values of DTQ$_{LQ}$. Middle row: DTQ$_{LQ}$ mesh points. Bottom row: Density values of a finely discretized DTQ$_{TR}$.
    Density values below $10^{-3}$ are shown as white for visualization purposes.}
    \label{fig:four hill}
\end{figure}

%\akil{What server? Just say that running such an experiment was not feasible on available hardware due to memory limitations.}

The simulation features a single mass breaking into four distinct hills and expanding outwards. The computed solutions of $\widehat{p}$ at various times are shown in \Cref{fig:four hill}. The comparison shows that DTQ$_{LQ}$ with a lenient discretization holds its own against a finely discretized DTQ$_{TR}$ simulation. Here, DTQ$_{LQ}$ utilized a spatial distance between points that was more than three times larger than the DTQ$_{TR}$ procedure and a temporal step size that was four times as large while still achieving very similar results. The average number of points used by DTQ$_{LQ}$ was about 2,000 per time step, and the finely discretized DTQ$_{TR}$ simulation used over 120,000 per time step. If a less finely discretized grid was used for DTQ$_{TR}$, it would need to span about $[-13, 13] \times [-13, 13]$. Assuming the same spacing used by the DTQ$_{LQ}$ and using $\kappa = 0.25$, the coarser DTQ$_{TR}$ simulation would still require about 11,000 points for such a grid, which is still significantly larger than the DTQ$_{LQ}$ mesh size. Additionally, \textit{a priori} knowledge of the necessary grid domain size is needed for DTQ$_{TR}$, which depends on the end time.
%\akil{$R$ should be subscript above.}

In this simulation, the percent of Leja points reused from the previous time step averaged about 84\% per time step. The alternative procedure from \cref{AltMethodSection} was used for approximately 1.2\% of the mesh on average per time step. The computation of these metrics is discussed in \cref{sec:metrics}.

\Cref{fig:pieCharts} \hl{shows the breakdown of time spent in the DTQ$_{LQ}$ algorithm for the erf function drift example. } \Cref{fig:pieCharts}(left) \hl{shows the high level breakdown between the quadrature rule and the mesh updates. } \Cref{fig:pieCharts}(right) \hl{gives the breakdown of the time spent computing the different pieces of the quadrature rule.}

\begin{table}[h!]
     \centering
     \begin{tabular}{ll}
\begin{tikzpicture}[scale=0.5,line join=round]
  \pie [/tikz/every pin/.style={align=center},
        text=pin,rotate = 155, color = {
        orange, 
        teal!20,
        magenta!40},
    hide number]
    {67/ Quadrature\\ rule,
    28/Adding\\ points\\ (\cref{AddingBoundaryPoints}),
    5/Removing\\ points\\ (\cref{RemovingBoundaryPoints})
    }
\end{tikzpicture}
&
\begin{tikzpicture}[scale=0.5,line join=round]
  \pie [/tikz/every pin/.style={align=center},
        text=pin,rotate = 180, color = {
        yellow!90!black,
        green!60!black,
        blue!60, 
        red!70,
        gray!70, 
        teal!20},
    hide number]
    {38.6/Laplace\\ approximation\\ (\cref{QuadFit}),
    31.4/Quadrature\\ weights\\ (\cref{ssec:iquad}),
    25/Leja points\\ (\cref{LejaPointsSection}),
    5/Alternative\\ method\\ (\cref{AltMethodSection})
    }
\end{tikzpicture}
\end{tabular}

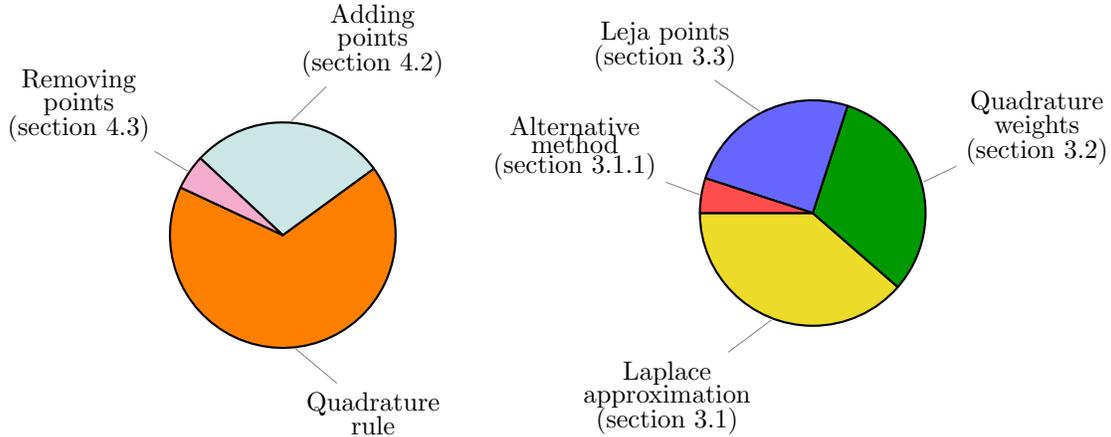
\captionof{figure}{\label{fig:pieCharts}
\hl{Breakdown of DTQ$_{LQ}$ algorithm for the erf function drift example in} \cref{ssec:four-hills}. \hl{(Left) Breakdown of time spent between approximating the Chapman-Kolmogorov equation and mesh updates. (Right) Breakdown of the algorithm components which make up the ``Quadrature rule'' slice of the left chart.}}
\label{samples}
\end{table}

\subsubsection{Spiral Example}
\label{ssec:spiral}
We now consider the solution to the SDE in \cref{eqn:SDE} in two dimensions where
\begin{align*}
  \mathbf{f}(\mathbf{x}) &=
\frac{5}{||\mathbf{x}||_2+10}\begin{pmatrix}
4\text{erf}(5x^{(1)})+2x^{(2)}\\
-2x^{(1)}+x^{(2)})
\end{pmatrix}, & 
\mathbf{g}(\mathbf{x}) = \begin{pmatrix}
0.6 & 0 \\
0 & 0.6\\
\end{pmatrix},
\end{align*}
with DTQ$_{LQ}$ simulation parameters $(\Delta{min}, \Delta{max}, \mathcal{R}, h, \beta) = (0.2, 0.2, 2, 0.02, 4)$ and the finely discretized DTQ$_{TR}$ simulation parameters were ($\kappa, h) = (0.05, 0.01)$. The solution to this problem features a single mass splitting into two and rotating in a clockwise spiral. 

\begin{figure}[h!]
    \centering
    \includegraphics[scale=0.55]{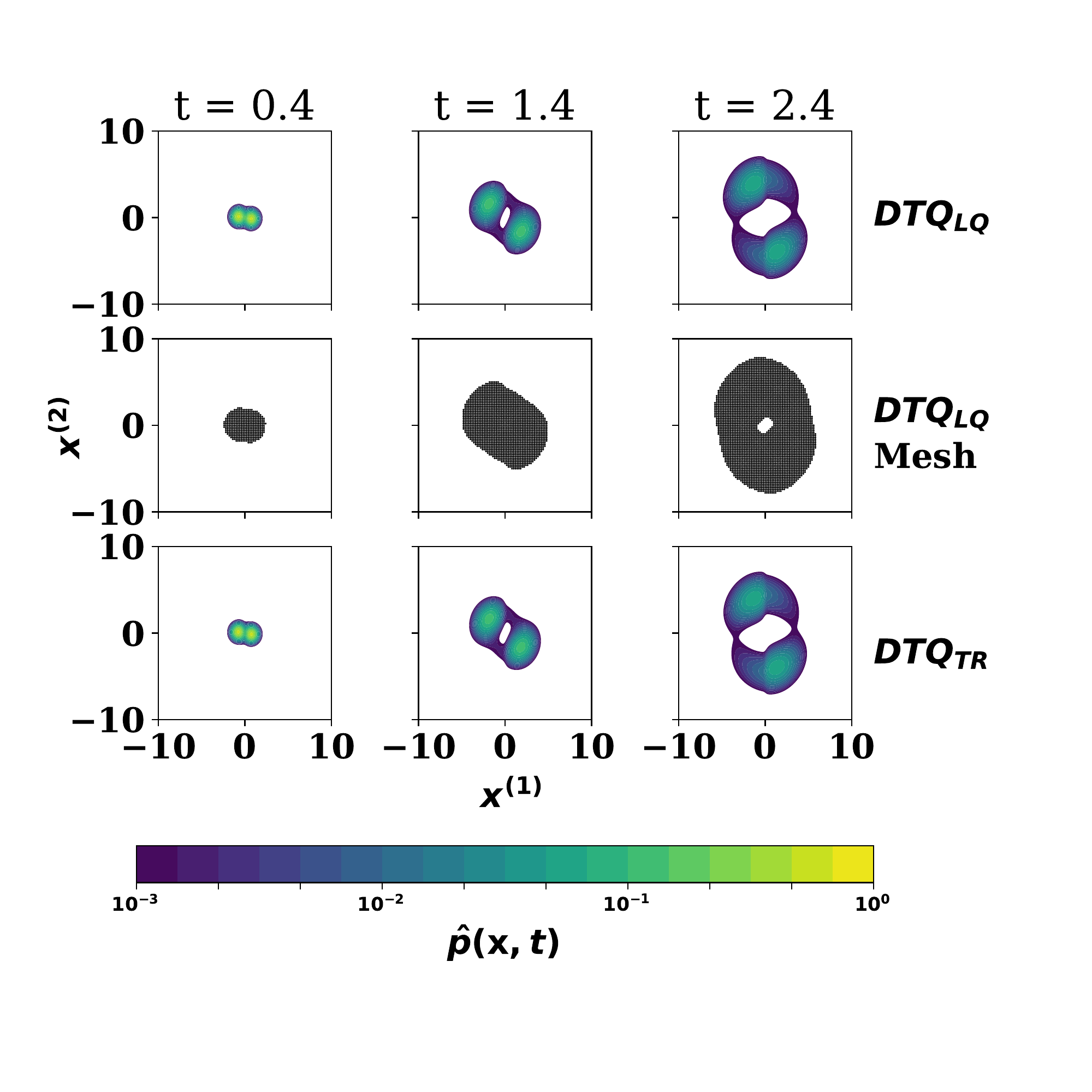}
    \caption{\textbf{Spiral example.} 
    The solution at the indicated times for the spiral drift example of \cref{ssec:spiral}. Top row: Density values of DTQ$_{LQ}$. Middle row: DTQ$_{LQ}$ mesh points. Bottom row: Density values of a finely discretized DTQ$_{TR}$.
    Density values below $10^{-3}$ are shown as white for visualization purposes.}
    \label{fig:Spiral}
\end{figure}

On average, DTQ$_{LQ}$ reused Leja points approximately 83\% of the time per time step, and the alternative method was used about 0.01\% of the time. The computation of these metrics is discussed in \cref{sec:metrics}. The computed solutions at various times are shown in \Cref{fig:Spiral}. We again see very similar results between the two procedures.

\subsubsection{Nonconstant Diffusion}\label{ssec:complex}
Now we consider the SDE in \cref{eqn:SDE} in two dimensions with a nonconstant diffusion and drift defined as
\begin{align*}
  \mathbf{f}(\mathbf{x}) &=
  \begin{pmatrix}
  2\text{erf}(10x^{(1)})\\
  0
  \end{pmatrix}, & 
  \mathbf{g}(\mathbf{x}) &= \begin{pmatrix}
  0.01(x^{(1)})^2+0.7 & 0.2 \\
  0.2 & 0.01(x^{(2)})^2+0.7\\
  \end{pmatrix}.
\end{align*}

\begin{figure}[h!]
    \centering
    \includegraphics[scale=0.55]{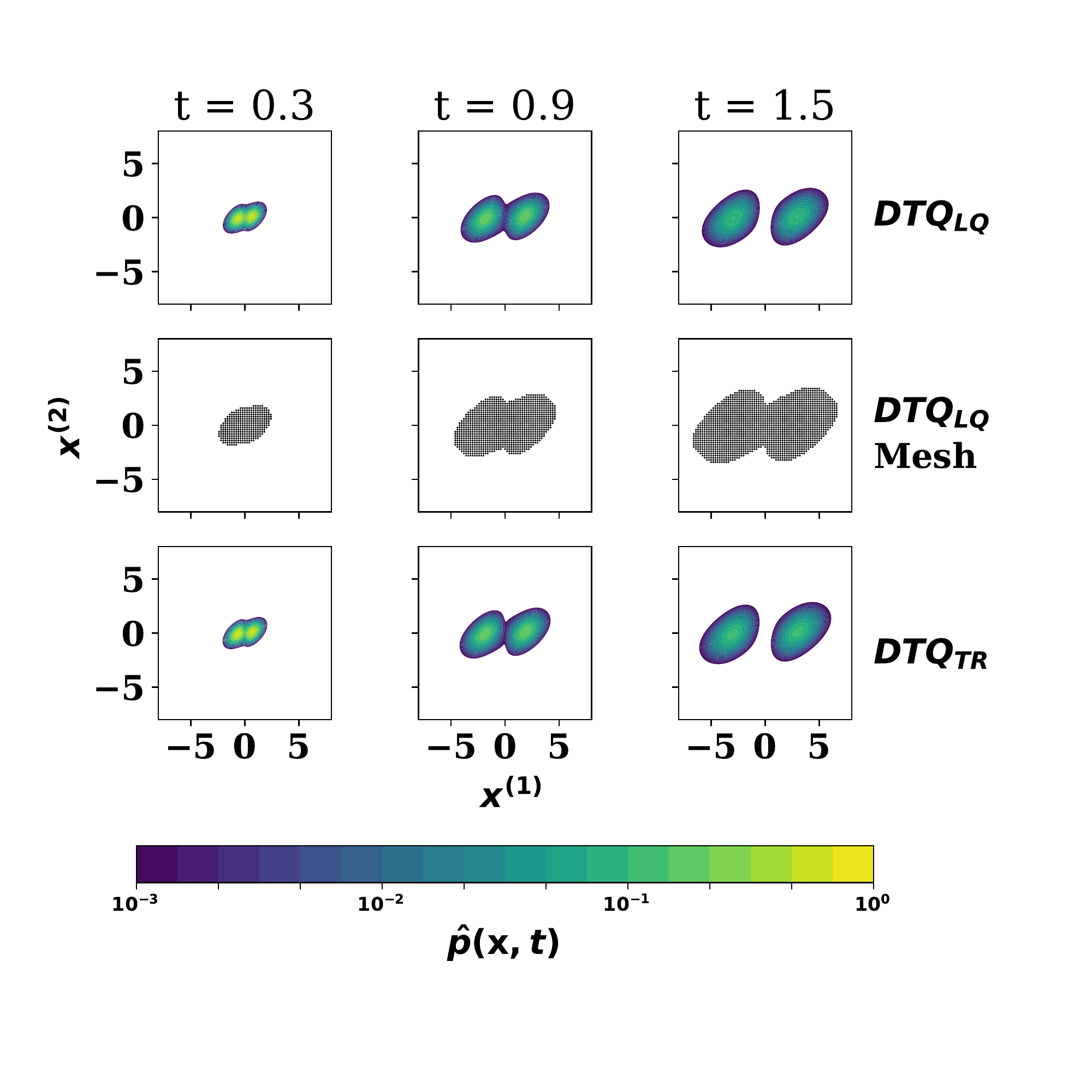}
    \caption{\textbf{Nonconstant diffusion example.} The solution at the indicated times for the nonconstant drift example of \cref{ssec:complex}. Top row: Density values of DTQ$_{LQ}$. Middle row: DTQ$_{LQ}$ mesh points. Bottom row: Density values of a finely discretized DTQ$_{TR}$.
    Density values below $10^{-3}$ are shown as white for visualization purposes.}
    \label{fig:complexDiff}
\end{figure}
This is the first example with a spatially dependent diffusion term. The DTQ$_{LQ}$ simulation parameter values were $(\Delta{min}, \Delta{max}, \mathcal{R}, h, \beta) = (0.2, 0.2, 2, 0.02, 4)$ and the finely discretized DTQ$_{TR}$ simulation parameters were ($\kappa, h) = (0.05, 0.01)$. On average, Leja points were reused approximately 87\% per time step, and the alternative method was used for 0\% of the mesh points for each time step. The computation of these metrics is discussed in \cref{sec:metrics}. The solution at various times is shown in \Cref{fig:complexDiff}.

 This example was included to illustrate that DTQ$_{LQ}$ can simulate a spatially dependent drift. We see very similar results between DTQ$_{LQ}$ and a finely discretized DTQ$_{TR}$.  
 
\subsection{Higher Dimensional Examples}\label{HigherDimensionalExample}
\subsubsection{Constant Drift and Diffusion}
\hl{We now consider the solution to the SDE in} \cref{eqn:SDE} \hl{with constant drift and diffusion in higher dimensions. For $N$ dimensions, we take}
\begin{align*}
  \mathbf{f}(\mathbf{x}) &= C_1\mathbf{e}_1, & 
    \mathbf{g}(\mathbf{x}) &= C_2\mathbf{I}_{N\times N}
\end{align*}
\hl{where $\mathbf{I}_{N\times N}$ is defined as the the $N \times N$-dimensional identity matrix and $\mathbf{e}_{1}$ is the cardinal $N$-dimensional unit vector which has a 1 as the first element and zeros elsewhere. This SDE corresponds to the Fokker-Planck PDE}
$$\frac{\partial p(\mathbf{x},t)}{\partial t}=-C_1\frac{\partial}{\partial x}p(\mathbf{x},t)+\sum_{i=1}^N\frac{C_2^2}{2}\frac{\partial^2}{\partial (x^{(i)})^2}p(\mathbf{x},t)$$ 
which has the exact solution  $$p(\mathbf{x},t)=\left(\frac{1}{4\pi (C_2^2/2)t}\right)^{N/2}\exp\left(\frac{-((x^{(1)}-C_1t)^2+\sum_{i=2}^N(x^{(i)})^2)}{4(C_2^2/2)t}\right).$$

\hl{For these examples, we take $C_1 = 1$ and $C_2=0.6$. In the three-dimensional example, we set simulation parameters as $(\Delta{min}, \Delta{max}, \mathcal{R}, h, \beta) = (0.22, 0.22, 1, 0.02, 3)$, in the four-dimensional example we set simulation parameters as $(\Delta{min}, \Delta{max}, \mathcal{R}, h, \beta) = (0.18, 0.18, 0.8, 0.02, 3)$, and in the the five-dimensional example we set simulation parameters as $(\Delta{min}, \Delta{max}, \mathcal{R}, h, \beta) = (0.1, 0.1, 0.5 , 0.01, 3)$. The results for these simulations are in} \Cref{tab:moving-hillHighDim}.

\begin{table}[h!]
\caption{\textbf{\hl{Constant drift and diffusion error for $N=3, 4, 5$ dimensions}}}\label{tab:moving-hillHighDim}
    \centering
    \begin{threeparttable}
        \begin{tabular}{c c c c c c}
            \toprule
            \midrule
            \multicolumn{1}{c}{$N$} & End Time & $h$& \mc{$L_{2p}$} Error &Starting Mesh Size & Ending Mesh Size\\
            \midrule
            3& 1& 0.02 & 0.00014 &  437 & 4,144\\ 
            4& 0.5& 0.02 & 0.00017 & 2,041 & 38,678\\ 
            5& 0.04 &0.01 & 0.0086  & 16,875 & 38,089\\ 
        \hline 
         \hline 
        \end{tabular}
    \end{threeparttable}
\end{table}

\hl{The purpose of this example is to illustrate that we can use DTQ$_{LQ}$ successfully in $N\geq 3$ dimensions. While we are able to demonstrate feasibility up to $N=5$ dimensions, computational constraints on memory play a role for such large dimensions. We estimate that the mesh size for DTQ$_{TR}$ in five dimensions that accurately covers the necessary support would be around 370,000 points, assuming a spatial grid spacing of 0.1 on a domain spanning from about $-0.6$ to $0.6$ in each dimension. }
%\akil{Meeting: grid size}
%1/(4pi*(0.6^2/2)*0.04)^(5/2)*exp(-r^2/(4*(0.6^2/2)*0.04)) < 0.001

\subsubsection{Three-dimensional Error Function Drift}
\hl{We consider the solution to the SDE in} \cref{eqn:SDE} \hl{in three dimensions with drift and constant diffusion}
\begin{align*}
  \mathbf{f}(\mathbf{x}) &= \left(\begin{array}{c} 2\text{erf}(10x^{(1)}) \\
  2\text{erf}(10x^{(2)})\\
  2\text{erf}(10x^{(3)}) \\
  \end{array}\right), & 
    \mathbf{g}(\mathbf{x}) &= 
    \left(\begin{array}{ccc} 
    0.75 & 0 & 0\\
    0 & 0.75 &0\\ 
    0 & 0& 0.75\end{array}\right).
    %\label{example:Erf3D}.
\end{align*}
\hl{In this three-dimensional erf drift example, we use DTQ$_{LQ}$ simulation parameters defined as $(\Delta{min}, \Delta{max}, \mathcal{R}, h, \beta)$ $= (0.25, 0.25, 1, 0.02, 3)$. We visualize the density using a color scale from white to black. The darker the color, the lager the density. The solution at time $t=1.22$ is shown in }\Cref{fig:3derf}. \hl{We note that the adaptive mesh has broken into eight different clusters in order to capture areas of high density without needing mesh points in areas of low density.}
%\akil{Meeting:New figure}
\begin{figure}[h!]
    \centering
    \includegraphics[scale=0.8]{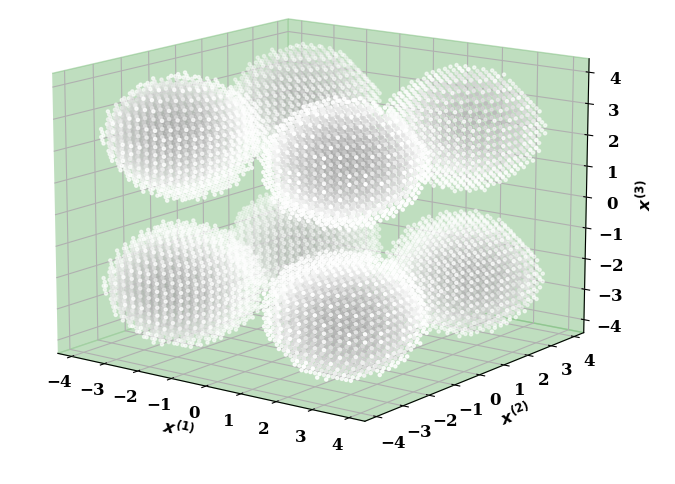}
    \caption{\textbf{\hl{Three-dimensional erf drift example.} The density is shown at $t=1.22$. The three spatial dimensions are plotted and the density is represented by the color of each point with the darker colors representing larger density.}}
    \label{fig:3derf}
\end{figure}

\section{Conclusion}
We have implemented an accurate and adaptive $N$-dimensional DTQ solver that uses fewer mesh points than current trapezoidal DTQ procedures and does not require \textit{a priori} knowledge of the necessary mesh coverage and domain. We included examples of the method when applied to problems from one to five dimensions. The results in \Cref{fig:T40Time} and \Cref{fig:T40Points} show that our adaptive DTQ procedure can be significantly more efficient while using substantially fewer mesh points when compared to DTQ using the trapezoidal quadrature rule on a tensorized grid.
\newpage

\begin{footnotesize}
\printbibliography[heading=none]
\end{footnotesize}
\end{document}